\documentclass[12pt,twoside,final]{article}

\usepackage{chngpage}
\usepackage[a4paper,left=22mm,right=22mm,top=30mm,bottom=30mm]{geometry}
\usepackage{imakeidx}
\usepackage{xr-hyper}
\usepackage{xr}
\usepackage[all]{xy}
\usepackage[center]{titlesec}
\titleformat*{\section}{\large\bfseries}
\usepackage{enumerate,cite}
\usepackage[latin1]{inputenc}
\usepackage{amssymb}
\usepackage{mathbbold}
\usepackage{amsthm}
\usepackage[leqno]{amsmath}
\usepackage{amsfonts}
\usepackage{mathrsfs}
\usepackage{hyperref}
\usepackage{color}
\usepackage{graphicx}
\usepackage{setspace}
\usepackage{bbm}
\usepackage{enumitem}
\usepackage{booktabs}
\usepackage[capitalize]{cleveref}
\usepackage{tikz}
\usetikzlibrary{matrix,chains}
\usepackage[center]{titlesec}
\titleformat*{\section}{\normalsize\bfseries}
\usepackage{indentfirst}
\usepackage{amscd}
\usepackage{hyperref}
\usepackage{mathbbold}
\theoremstyle{remark}

\theoremstyle{definition}

\theoremstyle{plain}
\newtheorem{thm}{Theorem}[section]

\newtheorem{lem}[thm]{Lemma}

\newtheorem{cor}[thm]{Corollary}

\newtheorem{prop}[thm]{Proposition}
\newtheorem{rmk}[thm]{Remark}
\newtheorem{order}[thm]{}

\newtheorem{conj*}{Conjecture}
\numberwithin{equation}{thm}

\usepackage{lastpage}
\usepackage{fancyhdr}

\begin{document}
\begin{center}{\Large\bf
Extensions of inertial blocks}

\bigskip{Kun Zhang}

\smallskip{\scriptsize Faculty of Mathematics and Statistics,
Hubei University, Wuhan, 430062, P.R. China }

\bigskip{Yuanyang
Zhou}

\smallskip{\scriptsize Department of Mathematics and
Statistics, Central China Normal University,
Wuhan, 430079, P.R. China}
\end{center}

\bigskip\noindent{\bf Abstract}\quad In this paper, with suitable assumptions, we generalize the work of K\"ulshammer and Puig on extensions of nilpotent blocks to inertial blocks.



\section  {Introduction}

\begin{order} {\rm Throughout this paper, $p$ is a prime number and ${\cal O}$ is a complete discrete valuation ring with an algebraically closed
residue field $k$ of characteristic $p$. Let $H$ be a finite group, $b$ a block of $H$
over ${\cal O}$ and $Q$ a defect group of the block $b$.
Then ${\cal O} Hb$ is the corresponding block algebra. Set ${\mathbbm H} =N_H(Q)$ and let ${\mathbbm b}$ be the Brauer
correspondent of $b$ in ${\mathbbm H}$. Let $P$ be a maximal $p$-subgroup of $G$ such
that ${\rm Br}_P^{{\cal O} H}(b)\neq 0$ (see paragraph \ref{Notation7} for ${\rm Br}_P^{{\cal O} H}$). By \cite[Proposition 5.3]{KP},
the intersection $P\cap H$ is a defect group of $b$ in $H$. So we may
assume $Q=P\cap H$. Let $G$ be a finite group containing $H$ as a normal subgroup. We always assume that $G$ stabilizes the block $b$ by conjugation.
Then ${\cal O} Gb$ is an ${\cal O}$-algebra with the identity $b$ and has an obvious $G/H$-graded algebra structure; we call it the extension of the block $b$. Block extensions
have been studied extensively and play an important role in representation theory of finite groups.}
\end{order}

\begin{order}{\rm
Assuming that the block $b$ of $H$ is nilpotent, K\"ulshammer and Puig characterized the algebraic structure of the block extension ${\cal O} Gb$ (see \cite[Theorems 1.8 and 1.12]{KP}).  Later Puig and the second author strengthened \cite[Theorem 1.12]{KP}
to get \cite[Corollary 3.15]{PZ2}, by which, they concluded that when the block $b$ is nilpotent,
there is an ${\cal O} K$-module such that its restriction to $H\times {\mathbbm H}$ induces a basic Morita equivalence between ${\cal O} Hb$ and ${\cal O} {\mathbbm H}  {\mathbbm b}$.
Here, setting ${\mathbbm G}=N_G(Q)$, then $K$ is equal to $(H\times{\mathbbm H})\Delta({\mathbbm G})$, where $\Delta({\mathbbm G})$ is the diagonal subgroup of ${\mathbbm G}\times {\mathbbm G}$.
In this paper, we try to generalize the results above to inertial blocks (see \cite{P11}), with some assumptions.}
\end{order}

\begin{thm}\label{Main1} Assume that $P$ is abelian and that $b$ as a block of $PH$ is inertial. Then there is an ${\cal O} K$-module $M$ such that its restriction
${\rm Res}^K_{H\times {\mathbbm H}  }(M)$ induces a basic Morita equivalence between ${\cal O} Hb$ and ${\cal O} {\mathbbm H}  {\mathbbm b}$; moreover the induced module ${\rm Ind}_K^{G\times {\mathbbm G} }(M)$ induces a Morita equivalence between ${\cal O} G b$ and ${\cal O} {\mathbbm G} {\mathbbm b}$ and any indecomposable direct summand of ${\rm Ind}_K^{G\times {\mathbbm G} }(M)$ induces a basic Morita equivalence between the corresponding block algebras.
\end{thm}

\begin{cor}\label{Main2}
Keep the notation as above and assume that $P$ is abelian and that  the block $b$  is inertial as a block of $PH$.
Let $N$ be a normal subgroup of $G$ containing $H$ such that $G=PN$ and $N/H$ is a $p'$-group.
Then the blocks of $G$ covering $b$ are inertial.
\end{cor}

\section  {Notation and terminology}

\begin{order}\label{Notation1} {\rm Throughout this paper, all
${\cal O}$-modules are ${\cal O}$-free finitely generated; all ${\cal O}$-algebras
have identity elements, but their subalgebras need not be unity ones. Let $\cal A$ be an
${\cal O}$-algebra; we denote by ${\cal A}^\circ\,,$ $Z({\cal A})$, $J({\cal
A})$, $1_{\cal A}$, ${\cal A}^*$ and ${\rm Aut}({\cal A})$ the opposite ${\cal O}$-algebra of ${\cal A}\,,$ the center of~${\cal A}$, the radical of
${\cal A}$, the identity element of ${\cal A}$, the multiplicative group of ${\cal A}^*$ and the group of all ${\cal O}$-automorphisms of $\cal A$, respectively. Sometimes we write $1$ instead of $1_{\cal A}$ without confusion. Let
${{{\cal B}}}$ be an ${\cal O}$-algebra; a homomorphism ${\cal F}: {\cal
A}\rightarrow {{{\cal B}}}$ of ${\cal O}$-algebras is said to be an {embedding\/}
if ${\cal F}$ is injective and
 ${\cal F}({\cal A})={\cal F}(1_{\cal A}){{{\cal B}}}{\cal F}(1_{\cal
A})$. Let $X$ be a finite group, ${\cal C}$ be an $X$-graded algebra and ${\cal C}_x$ be the $x$-component of ${\cal C}$ for any $x\in X$. The tensor product ${\cal A}\otimes_{\cal O} {\cal C}$
is an $X$-graded algebra with the $x$-component
${\cal A}\otimes_{\cal O} {\cal C}_x$ for any $x\in X$.}
\end{order}

\begin{order}\label{Notation2}  {\rm Let $X$ and $Y$ be groups with $Y$ normal in $X$. An ${\cal O}$-algebra ${{{\cal B}}}$ is a $Y$-interior
$X$-algebra (see ${\cal O} Y$-interior $X$-algebra in \cite{FP}) if there are group
homomorphisms $\varphi: Y\,\rightarrow\,{{{\cal B}}}^* $ and $\psi:\,
X\,\rightarrow\,{\rm Aut}({{{\cal B}}}) $ such that $\varphi$ preserves
the action of $X$ on $Y$ and that of $X$ on ${{{\cal B}}}^*$ and such that $\psi$ lifts
the conjugation action of $Y$ on ${{{\cal B}}}$. For any $x, y\in
Y$, any $z\in Z$ and any $a\in {{{\cal B}}}$, we write
$\varphi(x)a\varphi(y)$ as $xay$ and $\psi(z^{-1})(a)$ as $a^z$.
We call ${{{\cal B}}}$ $X$-interior algebras when $X=Y$ and $X$-algebras when $X=1$. Note that $X$-interior algebras here are the same as interior $X$-algebras in \cite{P1} and as ${\cal O} X$-interior algebras in \cite{P3}. Let ${{{\cal B}}}'$ be another $Y$-interior
$X$-algebra. The tensor product ${\cal
B}\otimes_{\cal O} {{{{\cal B}}}}'$ is a $Y$-interior $X$-algebra with the group
homomorphism
$ Y\rightarrow ({{{{\cal B}}}}\otimes_{\cal O}
{{{{\cal B}}}}')^*,\,y\mapsto y1_{{{\cal B}}}\otimes y1_{{{{\cal B}}}'}$ and with the action of $X$ on ${{{\cal B}}}\otimes_{\cal O}
{{{\cal B}}}'$ defined by the equality
$(a\otimes a')^x=a^x\otimes a'^x$ for any $a\in {{{\cal B}}}$, any $a'\in {{{\cal B}}}'$ and any $x\in X$. An
${\cal O}$-algebra homomorphism ${\cal F}:{{{\cal B}}}\rightarrow {{{\cal B}}}'$ is said
to be a $Y$-interior $X$-algebra homomorphism (see \cite[2.4]{P3})
if for any $x, y\in Y$, any $z\in Z$ and any $a\in {{{\cal B}}}$, we have
${\cal F}(xay)=x{\cal F}(a)y$ and ${\cal
F}(a^z)={\cal F}(a)^z$.
We denote by ${{{{\cal B}}}'}^X$ the ${\cal O}$-subalgebra of all $X$-fixed
elements in ${{{{\cal B}}}'}$. For any $a\in {{{\cal B}}}'^X$, the $a$-conjugation induces a $Y$-interior $X$-algebra isomorphism ${\rm int}(a):{{{\cal B}}}'\cong {{{\cal B}}}'$. The exomorphism of $Y$-interior $X$-algebras determined by $F$ is the set $\{{\rm int}(a)\circ F|a\in {{{{\cal B}}}'}^X\}$. We
call $\cal F$ an $X$-interior algebra homomorphism when $X=Y$ and
an $X$-algebra homomorphism when $X=1$.}
\end{order}

\begin{order}\label{Notation3}{\rm Let $i$ be an $X$-fixed idempotent in $Y$-interior
$X$-algebra ${{{\cal B}}}$.
It is easy to check that $i{{{\cal B}}} i$ is a $Y$-interior $X$-algebra with
the group homomorphism $Y\rightarrow (i {{{\cal B}}} i)^*, \, y\mapsto yi $
and with the group homomorphism $X\rightarrow {\rm Aut}(i{{{\cal B}}} i)$
mapping $x$ onto the restriction of $\psi(x)$ to $i{{{{\cal B}}}} i$ for any
$x\in X$. Let $W$ be a subgroup of $X$. Clearly ${{{\cal B}}}$ is a
$(W\cap Y)$-interior $W$-algebra with the restrictions of $\varphi$
to $W\cap Y$ and of $\psi$ to $W$, denoted by ${\rm Res}^X_W({{{\cal B}}})$. Let
$\cal C$ be an $X$-interior algebra with the structure homomorphism
$\rho: X\rightarrow {\cal C}^*$. Let $\varrho: Z\rightarrow X$ be a
group homomorphism. The algebra ${\cal C}$ with the group
homomorphism $\rho\circ\varrho: Z\rightarrow {\cal C}^*$ is a
$Z$-interior algebra, denoted by ${\rm Res}_{\varrho}({\cal C})$.}
\end{order}

\begin{order} \label{Notation6}{\rm A $k^*$-group
is a group $\hat X$ with an injective group
homomorphism $\theta: k^*\rightarrow Z(\hat X)$, where $Z(\hat X)$ is the center of $\hat X$, and the quotient $\hat X/\theta(k^*)=X$ is called the $k^*$-quotient of $\hat X$; usually we omit to mention $\theta$ and write $\lambda x$ instead of $\theta(\lambda) x$ for any $\lambda\in k^*$ and any $x\in \hat X$. By \cite[Charpter II, Proposition 8]{S} there is a canonical decomposition ${\cal O}^*\cong k^*\times (1+J({\cal O}))$, through which we regard $k^*$ as a subgroup of ${\cal O}^*$. We define the twisted group algebra ${\cal O}_*\hat
X$ to be the ${\cal O}$-algebra ${\cal O}\otimes _{{\cal O} k^*} {\cal O} \hat X$ with the multiplication induced by the product in $\hat X$. Let $\hat U$ be another $k^*$-group with $k^*$-quotient
$U$. A group homomorphism $\phi: \hat X\rightarrow \hat U$ is a $k^*$-group
homomorphism if $\phi(\lambda x)=\lambda\phi(x)$
for any $\lambda\in k^*$ and $x\in \hat X$. We denote by $\hat V$ the inverse image of $V$ in $\hat X$ for any subset $V$ of $X$. If $V$ is a $p$-subgroup of $X$, by \cite[2.10]{P6} there is
a unique $k^*$-group isomorphism $\hat V\cong k^*\times V$. Therefore we can identify $V$ as a subgroup of $\hat X$ through this $k^*$-group isomorphism. We denote by $\hat X^\circ$ the $k^*$-group with the same underlying group  $\hat X$ endowed with the group homomorphism
$\theta^{-1}: k^*\rightarrow Z(\hat X),\, \lambda\mapsto \theta(\lambda)^{-1}$.
Let $\vartheta: Z\rightarrow X$ be a group homomorphism; we denote by ${\rm Res}_{\vartheta}(\hat X)$ the $k^*$-group formed by all pairs
$(y, x)\in Z\times \hat X$ such that $\vartheta(y)$ is the
image of $x$ in $X$, endowed with the group homomorphism mapping $\lambda\in k^*$ on $(1, \lambda)$. The $k^*$-group ${\rm Res}_{\vartheta}(\hat X)$ has $k^*$-quotient isomorphic to $Z$ and thus is a $k^*$-group with $k^*$-quotient $Z$. For more on
$k^*$-groups, see \cite[\S 5]{P6}.}
\end{order}

\begin{order} \label{Notation7}{\rm Let ${{{\cal B}}}$ be an
$X$-algebra over ${\cal O}$. For any subgroup $Z$ of
$X$, we denote by ${{{\cal B}}}^Z$ the ${\cal O}$-subalgebra of all $Z$-fixed
elements in ${{{\cal B}}}$. For any subgroup $U$ of $X$, we set
${{{\cal B}}}(U)=k\otimes _{\cal O} ({{{\cal B}}}^U/\sum_V {{{\cal B}}}^U_V),$ where $V$ runs over the set of proper subgroups of $U$ and ${{{\cal B}}}^U_V$ is the image
of the relative trace map ${\rm Tr}_V^U: {{{\cal B}}}^V\rightarrow {{{\cal B}}}^U$. We denote by ${\rm Br}^{{{\cal B}}}_U$ the canonical surjective homomorphism ${\rm Br}^{{{\cal B}}}_U:
{{{\cal B}}}^U\rightarrow {{{\cal B}}}(U)$. When ${{{\cal B}}}$ is a $Y$-interior $X$-algebra, the $Y$-interior $X$-algebra on ${{{\cal B}}}$ induces $C_Y(U)$-interior $N_X(U)$-algebra structures on ${{{\cal B}}}(U)$ and on ${{{\cal B}}}^U$ and then ${\rm Br}^{{{\cal B}}}_U$ is a $C_Y(U)$-interior $N_X(U)$-algebra homomorphism. When ${{{\cal B}}}$ is equal to the group
algebra ${\cal O} X$, the $k$-algebra homomorphism $kC_X(U)\rightarrow {{{\cal B}}}(U)$
sending $x\in C_X(U)$ onto the image of $x$ in ${{{\cal B}}}(U)$ is a $k$-algebra
isomorphism (see \cite[Proposition 37.5]{T}), through which we identify ${{{\cal B}}}(U)$ with
$kC_X(U)$.}
\end{order}

\begin{order} \label{Notation8}{\rm Recall that a pointed group $V_\beta$ on ${{{\cal B}}}$
consists of a subgroup $V$ of $X$ and a $({{{\cal B}}}^V)^*$-conjugacy
class $\beta$ of primitive idempotents of ${{{\cal B}}}^V$. We also say that
$\beta$ is a {point\/} of $V$ on ${{{\cal B}}}$. Obviously $X$ acts on the
set of all pointed groups on ${{{\cal B}}}$ by the equality $(V_\beta)^x=V^x_{\beta^x}$. We denote by
$N_X(V_\beta)$ the stabilizer of $V_\beta$ in $X$. Another pointed group $Z_\gamma$ is
contained in $V_\beta$ and we write $Z_\gamma\leq V_\beta$ if $Z\leq V$ and there exist
$i\in \beta$ and $j\in \gamma$ such that $ij=ji=j$. A pointed group $U_\gamma$ on ${{{\cal B}}}$ is local if the
image of $\gamma$ in ${{{\cal B}}}(U)$ is not equal to $\{0\}$. A pointed group $U_\gamma$
is a defect pointed group of a pointed group $V_\beta$ on ${{\cal B}}$
if $U_\gamma$ is a maximal local pointed group contained in $V_\beta$. Since an
$X$-interior algebra structure induces an $X$-algebra structure by
the $X$-conjugation, the above terminology on pointed groups applies to $X$-interior algebras. For more on pointed groups, see \cite{P1} and \cite{T}.}
\end{order}

\begin{order} \label{Notation9}{\rm Set $\dot{X}=X/Y$ and denote by $\dot x$ the image
of $x$ in $\dot X$ for any $x\in X$. Let $Z$ be a subgroup of $X$.
Let ${\cal A}$ be an $\dot X$-graded $Z$-interior algebra such that the
image of $z$ in ${\cal A}$ is just inside the $\dot z$-component of ${\cal A}$
for any $z\in Z$. Then the
$\dot X$-graded $Z$-interior algebra structure on ${\cal A}$ induces a $W$-interior $Z$-algebra structure on the $1$-component ${{{\cal B}}}$ of the $\dot X$-graded algebra ${\cal A}$, where we set $W$ to be the intersection $Y\cap Z$. Let $R_\varepsilon$ be a local pointed group on ${{{\cal B}}}$.
Take some $h\in \varepsilon$ and set ${\cal A}_\varepsilon=h{\cal A} h$
and ${{{\cal B}}}_\varepsilon=h{{{\cal B}}} h$. Then it is easy to see that
${\cal A}_\varepsilon$ is an $\dot X$-graded $R$-interior algebra with the
$\dot x$-component $h{\cal A}_{\dot x}h$ for any $\dot x\in \dot X$ and with the group homomorphism
$R\rightarrow {\cal A}_\varepsilon^*,\,u\mapsto uh$ for any $u\in
R$, and that ${{{\cal B}}}_\varepsilon$ is a $(R\cap Y)$-interior $R$-subalgebra of ${\cal A}_\varepsilon$ (see paragraph \ref{Notation3}).}
\end{order}

\begin{order} \label{Notation10}{\rm For any $\dot{x}\in \dot{X}$, we denote by ${\cal N}_{{\cal
A}_\varepsilon^*}^{\dot{x}}(R)$ the set of all invertible elements
in the $\dot{x}$-component of ${\cal A}_\varepsilon$ normalizing
$Rh$. Notice that ${\cal N}_{{\cal A}_\varepsilon^*}^{\dot{x}}(R)$
may be empty. We set
${\cal N}_{{\cal A}_\varepsilon^*}(R)=\cup_{\dot x\in
\dot{X}} {\cal N}_{{\cal A}_\varepsilon^*}^{
\dot{x}}(R)$. Then it is easily checked that ${\cal N}_{{\cal
A}_\varepsilon^*}(R)$ is a group with respect to the multiplication
of $\cal A_\varepsilon$ and that $({{{\cal B}}}_\varepsilon^R)^*$ is a
normal subgroup of ${\cal N}_{{\cal A}_\varepsilon^*}(R)$. We set
${\cal
F}_{\cal A}(R_\varepsilon)={\cal N}_{{\cal A}_\varepsilon^*}(R)/({\cal
B}_\varepsilon^R)^*$ and $\hat {\cal
F}_{\cal A}(R_\varepsilon)={\cal N}_{{\cal A}_\varepsilon^*}(R)/(1+J({{{\cal B}}}_\varepsilon^R))$.
Since ${{{\cal B}}}_\varepsilon^R/J({\cal
B}_\varepsilon^R)\cong k$, the quotient group $\hat {\cal
F}_{\cal A}(R_\varepsilon)$ has a natural $k^*$-group structure with $k^*$-quotient ${\cal
F}_{\cal A}(R_\varepsilon)$. For any $a\in {\cal N}_{{\cal A}_\varepsilon^*}(R)$, we denote by $\bar a$ the image of $a$ in $\hat {\cal
F}_{\cal A}(R_\varepsilon)$. When $\dot X$ is the trivial group, ${\cal N}_{{\cal A}_\varepsilon^*}(R)$ is just the normalizer of $Rh$ in ${\cal A}_\varepsilon^*$; in this case, we write ${\cal N}_{{\cal A}_\varepsilon^*}(R)$ as ${N}_{{\cal A}_\varepsilon^*}(R)$.}
\end{order}

\begin{order}\label{Notation11} {\rm Now we generalize fusions in \cite{P2} to graded algebras. A pair $(\varphi, \,\dot x)$ is called an $({\cal A},\, \dot X)$-fusion of $R_\varepsilon$ if $\varphi$ is a group automorphism on $R$, $\dot x$ is an element of $\dot X$ and there is an element $a$ in ${\cal N}_{{\cal A}_\varepsilon^*}^{\dot{x}}(R)$ such that $aua^{-1}=\varphi(u)h$ for any $u\in R$. We denote by $F_{{\cal A},\, \dot X}(R_\varepsilon)$ the set of all $({\cal A},\, \dot X)$-fusions of $R_\varepsilon$. $F_{{\cal A},\, \dot X}(R_\varepsilon)$ is a group with respect to the composition determined by the composition of maps and by the product in $\dot X$. It is easy to prove that the following holds.}

\medskip\noindent
\hypertarget{prop2111}{\rm 2.9.1.}\quad {\it The map $\theta: F_{{\cal A},\, \dot X}(R_\varepsilon)\rightarrow {\cal F}_{\cal A}(R_\varepsilon)$ mapping $(\varphi, \,\dot x)$ to $\bar a$, where $a$ is an element of ${\cal N}_{{\cal A}_\varepsilon^*}^{\dot{x}}(R)$ such that $aua^{-1}=\varphi(u)h$ for any $u\in R$, is a group homomorphism. Moreover the homomorphism $\theta$ is a group isomorphism if the map $R\rightarrow Rh,\, u\mapsto uh$ is a group isomorphism.}

{\rm \medskip\noindent We set $\hat F_{{\cal A},\, \dot X}(R_\varepsilon)={\rm Res}_\theta(\hat {\cal
F}_{\cal A}(R_\varepsilon))$. That is to say, $\hat F_{{\cal A},\, \dot X}(R_\varepsilon)$ consists of all triples $(\varphi,\, \dot x,\,\bar a)$, where $a$ is an element of ${\cal N}_{{\cal A}_\varepsilon^*}^{\dot{x}}(R)$ such that $aua^{-1}=\varphi(u)h$ for any $u\in R$.
The group $F_{{\cal A},\, \dot X}(R_\varepsilon)$ has a normal group $F_{{{\cal B}}}(R_\varepsilon)$, which consists of all $({\cal A},\, \dot X)$-fusions $(\varphi,\, 1)$ of $R_\varepsilon$. When $\dot X$ is the trivial group, we write $\hat F_{{\cal A},\, \dot X}(R_\varepsilon)$, $F_{{\cal A},\, \dot X}(R_\varepsilon)$, $(\varphi,\,1,\, \bar a)$ and $(\varphi,\,1)$ as $\hat F_{{\cal A}}(R_\varepsilon)$, $F_{{\cal A}}(R_\varepsilon)$, $(\varphi,\, \bar a)$ and $\varphi$, respectively. Then $F_{{\cal A}}(R_\varepsilon)$ is just the usual fusion group in \cite{P2}.}
\end{order}

\section  {Extension category}

\begin{order}\label{Notation12} {\rm First we fix the setting of the remainder of the paper. Let $G$ be a finite group and $H$ a
normal subgroup of $G$. Let $b$ be a $G$-stable block of $H$ over ${\cal O}$. Set ${\cal A} = {\cal O} Gb$, ${{{\cal B}}} = {\cal O} Hb$ and $\dot G=G/H$. For any $x\in G$ and any subset $I$ of $G$, we denote by $\dot{x}$ and $\dot I$ the images of $x$ and $I$ in $\dot{G}$, respectively. Clearly $\cal A$ is a $\dot{G}$-graded algebra
with the $\dot{x}$-component ${\cal O} H x b$.
Obviously,
$\alpha=\{b\}$ is a point of both $H$ and $G$ on ${\cal O} H$.
Let $Q_\delta$ and $P_\gamma$ be defect pointed groups of~$H_\alpha$ and $G_\alpha$, respectively, such that $Q_\delta\leq P_\gamma\,.$ Choosing  $j\in \delta$
and $i\in \gamma$ such that $ij=ji=j\,,$ we set
${\cal A}_\gamma=i{\cal A} i$, ${{{\cal B}}}_\gamma=i{{{\cal B}}}i$ and
${{{\cal B}}}_\delta=j{{{\cal B}}}j$. Then ${\cal A}_\gamma$ is
a $\dot {G}$-graded $P$-interior algebra, ${\cal
B}_\gamma$ is a $Q$-interior $P$-subalgebra of ${\cal A}_{\gamma}$ and the $Q$-interior algebra
${{{\cal B}}}_\delta$ is a source algebra of the block algebra ${{{\cal B}}}$.
Set ${\mathbbm H}=N_H(Q)$ and ${\mathbbm G}=N_G(Q)$.
Let $\mathbbm b$ be the Brauer correspondent of $b$ in ${\mathbbm H}$ and set $\bbalpha=\{\mathbbm b\}$. Since $b$ is $G$-stable, $\mathbbm b$ is ${\mathbbm G}$-stable. Thus $\bbalpha$ is a point of both ${\mathbbm G}$ and ${\mathbbm H}$ on ${\cal O} {\mathbbm H}$.
We take the local points $\bbgamma$ and $\bbdelta$ of $P$ and $Q$ on ${\cal O} {\mathbbm H}$, respectively, such that ${\rm Br}_P^{{\cal O} H}(\gamma)={\rm Br}_P^{{\cal O} {\mathbbm H}}(\bbgamma)$ and ${\rm Br}_Q^{{\cal O} H}(\delta)={\rm Br}_Q^{{\cal O} {\mathbbm H}}(\bbdelta)$. Then it is easy to verify that $P_{\bbgamma}$ is a defect pointed group of ${\mathbbm G}_{\bbalpha}$, that $Q_{\bbdelta}$ is a defect pointed group of ${\mathbbm H}_{\bbalpha}$ and that $Q_{\bbdelta}$ is contained in $P_{\bbgamma}$.
Take $\mathbbm{i}\in \bbgamma$ and $\mathbbm{j}\in \bbdelta$. Set ${\mathbbm A}={\cal O} {\mathbbm G} \mathbbm b$, ${\mathbbm B}={\cal O} {\mathbbm H} \mathbbm b$, ${\mathbbm A}_\bbgamma=\mathbbm{i}{\mathbbm A}\mathbbm{i}$, ${\mathbbm B}_\bbgamma=\mathbbm{i}{\mathbbm B}\mathbbm{i}$ and ${\mathbbm B}_\bbdelta=\mathbbm{j}{\mathbbm B}\mathbbm{j}$.
For any local pointed group $R_\varepsilon$ such that $Q_\delta\leq R_\varepsilon\leq P_\gamma$, we set $C_G(R_\varepsilon)=N_G(R_\varepsilon)\cap C_G(R)$ and denote by $ b_\varepsilon$ the block of $C_H(R)$ over ${\cal O}$ such that ${\rm Br}_R^{{\cal O} H}(\varepsilon){\rm Br}_R^{{\cal O} H}(b_\varepsilon)={\rm Br}_R^{{\cal O} H}(\varepsilon)$.
Throughout this section, we assume
$P$ is abelian.
}
\end{order}

\begin{order} {\rm
 Let $R_\varepsilon$ and $T_\zeta$ be local pointed groups on ${\cal O} H$.
For any $x\in G$ such that ${}^x(R_\varepsilon)\leq T_\zeta$, we denote by $\varphi_{R,\, x}^T$ the group homomorphism $\varphi_{R,\, x}^T: R\rightarrow T, \,u\mapsto xux^{-1}$. We call the pair $(\varphi_{R,\, x}^T, \,\dot{x})$ a $(G,\,\dot G)$-homomorphism from $R_\varepsilon$ to $T_\zeta$, where we set $\dot G=G/H$ and $\dot x$ is the image of $x$ in $\dot G$. By \cite[Theorem 1.8]{KP} or \cite[Theorem 3.5]{PZ2}, we can see that local pointed groups $R_\varepsilon$ on ${\cal O} H$
such that $Q_\delta\leq R_\varepsilon\leq P_\gamma$ play an important role in determining the algebra structure of block extensions. So we define a category ${\cal E}_{(P_\gamma,\,H,\,G)}$, where the objects are local pointed groups $R_\varepsilon$ such that $Q_\delta\leq R_\varepsilon\leq P_\gamma$ and, for any pair of objects $R_\varepsilon$ and $T_\zeta$, the morphisms from $R_\varepsilon$ to $T_\zeta$ are $(G,\,\dot G)$-homomorphisms from $R_\varepsilon$ to $T_\zeta$; the composition in ${\cal E}_{(P_\gamma,\,H,\,G)}$ is determined by the composition of group homomorphisms and the product in $\dot G$. The category ${\cal E}_{(P_\gamma,\,H,\,G)}$ is the local structure of the block extension ${\cal O} Gb$.
Note that $\{1\}$ is the unique local point of $R$ on ${\cal O} Q$ for any subgroup $R$ of $P$.}
\end{order}

For any local pointed group  $R_\varepsilon$ on ${\cal O} H$,
we denote by $N_G(R_\varepsilon)$ the stabilizer of $R_\varepsilon$ in $G$ under the $G$-conjugation action.
Then we have the following lemma.

\begin{lem}\label{N_G(R)}  Let $R_\varepsilon$ be a local pointed group such that $R_\varepsilon\leq P_\gamma$ and $Q\subset R$. Then we have $N_G(R_\varepsilon)=N_{N_G(R_\varepsilon)}(P_\gamma) C_G(R_\varepsilon)$.
\end{lem}

\begin{proof} Since $P$ is abelian, by \cite[Proposition 5.5]{KP}, we have $P\subset C_G(R_\varepsilon)$. Clearly there is a pointed group $C_G(R_\varepsilon)_\eta$ on ${\cal O} H$ such that $P_\gamma$ is a defect pointed group of $C_G(R_\varepsilon)_\eta$. Since $R_\varepsilon\leq P_\gamma$, by \cite[2.16.2]{KP} the pointed group $C_G(R_\varepsilon)_\eta$ is unique. Clearly $N_G(R_\varepsilon)$ stabilizes $C_G(R_\varepsilon)_\eta$ and thus the set of all defect pointed groups of $C_G(R_\varepsilon)_\eta$. Since $C_G(R_\varepsilon)$ acts transitively on this set, by Frattini argument we have $N_G(R_\varepsilon)=N_{N_G(R_\varepsilon)}(P_\gamma) C_G(R_\varepsilon)$.
\end{proof}

\begin{order}\label{NotationF1}\label{L} {\rm Set ${\cal N}=N_G(Q_\delta)$, ${\cal H}=C_H(Q)$,  $\beta=\{b_\delta\}$ and set ${\cal E}={\cal N}/{\cal H}$. Clearly ${\cal N}_\beta$ and ${\cal H}_\beta$ are pointed groups on ${\cal O} {\cal H}$. Let $R_\varepsilon$ be a local pointed group such that $Q_\delta\leq R_\varepsilon\leq P_\gamma$ and let $R_\epsilon$ be a local pointed group on ${\cal O} {\cal H}$ such that ${\rm Br}_R^{{\cal O} H}(\varepsilon)={\rm Br}_R^{{\cal O} {\cal H}}(\epsilon)$. Obviously ${\rm Br}_Q^{{\cal O} {\mathbbm H}}(\epsilon)$ is a local point of $R$ on $k{\cal H}$. Since ${\rm Br}_Q^{{\cal O} {\mathbbm H}}$ maps $({\cal O} {\mathbbm H})^R$ onto $(k {\cal H})^R$ and ${\rm Ker}({\rm Br}_Q^{{\cal O} {\mathbbm H}})$ is contained in $J({\cal O} {\mathbbm H})$, $R_\epsilon$ has to be a local pointed group on ${\cal O} {\mathbbm H}$. In particular, $P_{\bbgamma}$ and $Q_\bbdelta$ are local pointed groups on ${\cal O} {\cal H}$. Moreover, $P_{\bbgamma}$ is a defect pointed group of ${\cal N}_\beta$ and $Q_\bbdelta$ is a defect pointed group of ${\cal H}_\beta$. Since the block $b_\delta$ of ${\cal H}$ is nilpotent (see \cite{BP}), by \cite[Theorem 3.5]{PZ2} there are a finite group
$L$ containing $P$ and a surjective group homomorphisms $\bbpi: L\rightarrow {\cal
E}$ with kernel $Q$ such that $\bbpi(u)$ is the image of $u$ in ${\cal E}$ for any $u\in P$ and such that with the identification of $L/Q$ and ${\cal E}$ through the isomorphism $L/Q\cong {\cal E}$ induced by $\bbpi$, the functor
$\bbtau:{\cal E}_{(P_{\bbgamma},\, {\cal H},\,{\cal N})}\rightarrow {\cal E}_{(P_{\{1\}},\, Q,\, L)}$ mapping an object $R_\bbespilon$ onto $R_{\{1\}}$ and a morphism $(\varphi_{R,\, x}^T, \,\tilde{x})$ onto $(\varphi_{R,\, x}^T, \,\pi(x))$ is an isomorphism of categories, where $\tilde{x}$ is the image of $x$ in ${\cal E}$.}
\end{order}

\begin{prop}\label{Cal-L} Keep the notation as above. There are a finite group
${\cal L}$ containing $P$ and a surjective group homomorphism $\pi: {\cal L}\rightarrow {\cal
E}$ with the kernel $Q$ such that $\pi(u)$ is the image of $u$ in ${\cal E}$ for any $u\in P$ and such that with the identification of ${\cal L}/Q$ and ${\cal E}$ through the isomorphism ${\cal L}/Q\cong{\cal E}$ induced by $\pi$, the functor $\tau: {\cal E}_{(P_\gamma,\,H,\,G)}\rightarrow {\cal E}_{(P_{\{1\}},\,Q,\,{\cal L})}$ mapping an object $R_\varepsilon$ onto $R_{\{1\}}$ and a morphism $(\varphi_{R,\, x}^T, \,\dot{x})$ onto $(\varphi_{R,\, x}^T, \,\pi(x))$ is an isomorphism of categories. Moreover for another such a pair of ${\cal L}'$ and $\pi'$, there is a group isomorphism $\varphi: {\cal L}\rightarrow {\cal L'}$, unique up to conjugation, such that $\pi'\circ\varphi=\pi$ and such that $\varphi(u)=u$ for any $u\in P$.
\end{prop}

 \begin{proof} Take any two local pointed groups $R_\varepsilon$ and $T_\nu$ in ${\cal E}_{(P_\gamma,\,H,\,G)}$ and let $x$ be an element of $G$ such that $(R_\varepsilon)^{x^{-1}}\leq T_\nu$.
By \cite[Proposition 5.3]{KP}, we have $Q^{x^{-1}}=(R^{x^{-1}}\cap H)=T\cap H=P\cap H=Q$. Furthermore, since $(Q_\delta)^{x^{-1}}\leq (R_\varepsilon)^{x^{-1}}\leq P_\gamma$, by \cite[Proposition 5.5]{KP} we have  $\delta^{x^{-1}}=\delta$ and thus $x\in N_G(Q_\delta)$. Let $R_\bbespilon$ and $T_\bbnu$ be local pointed groups on ${\cal O} {\cal H}$ such that ${\rm Br}_R^{{\cal O} H}(\varepsilon)={\rm Br}_R^{{\cal O} {\cal H}}(\bbespilon)$ and ${\rm Br}_T^{{\cal O} H}(\nu)={\rm Br}_R^{{\cal O} {\cal H}}(\bbnu)$, respectively.
Note that $(R_\varepsilon)^{x^{-1}}\leq T_\nu$ is equivalent to $(R_\bbespilon)^{x^{-1}}\leq T_\bbnu$. We define a functor $\psi: {\cal E}_{(P_{\gamma},\, H,\, G)}\rightarrow {\cal E}_{(P_{\bbgamma},\, {\cal H},\, {\cal N})}$ mapping an object $R_\varepsilon$ onto $R_\bbespilon$ and a morphism $(\varphi_{R,\, x^{-1}}^T, \,\dot x)$ from $R_\varepsilon$ to $T_\nu$ onto a morphism $(\varphi_{R,\, x^{-1}}^T, \,\tilde x)$ from $R_\bbespilon$ to $T_\bbnu$. It is trivial to check that this functor is an isomorphism of categories. Set $\tau=\bbtau\circ\psi$. Then $\tau$ is an isomorphism of categories fulfilling Proposition \ref{Cal-L}. Therefore $L$ and $\pi$ satisfy Proposition \ref{Cal-L}. Suppose that there are another finite group $L'$ and another group homomorphism $\pi':L'\rightarrow {\cal E}$ satisfying Proposition \ref{Cal-L}. We denote by $\tau'$ the corresponding isomorphism of categories $ {\cal E}_{(P_\gamma,\,H,\,G)}\cong {\cal E}_{(P_{\{1\}},\,Q,\,L')}$. Then $\tau'\circ \psi^{-1}$ is the isomorphism of categories ${\cal E}_{(P_{\bbgamma},\, {\cal H},\,{\cal N})}\cong{\cal E}_{(P_{\{1\}},\, Q,\, L')}$ induced by the inclusion $P\subset L'$ and the homomorphism $\bbpi$ in the sense of \cite[Theorem 3.5]{PZ2}. Then the last statement of Proposition \ref{Cal-L} follows from the uniqueness part of \cite[Theorem 3.5]{PZ2}.
\end{proof}

\begin{rmk}
{The finite group $L$  and the homomorphism $\bbpi$ in paragraph \ref{NotationF1}  can be chosen to be ${\cal L}$ and $\pi$, respectively.
For consistency, we shall use the symbols $\cal{L}$ and $\pi$ throughout.}
\end{rmk}

\begin{order}\label{Def-hatNGQ}{\rm
We denote by $U$ the simple
factor of $({\cal O} H)^Q$ determined by the point $\delta$ and by
$s_\delta$ the canonical surjective homomorphism $({\cal O} H)^Q\rightarrow U$. Explicitly $U$ is the simple factor of $({\cal O} H)^Q$ such that $s_\delta(\delta)\neq \{0\}$.
Clearly the $N_G(Q_\delta)$-conjugation stabilizes $\delta$ and thus $U$.  We construct the subgroup $\hat N_G(Q_\delta)$ of the direct product $N_G(Q_\delta)\times U^*$ consisting of all elements $(x, s_\delta(a))$  such that the $x$- and $s_\delta(a)$-conjugations have the same action on $U$.
Clearly there are injective group homomorphisms $k^*\rightarrow \hat N_G(Q_\delta),\, \lambda\mapsto (1,\,\lambda)$ and $C_H(Q)\rightarrow \hat N_G(Q_\delta),\,
x\mapsto (x, s_\delta(x))$. We identify $k^*$ and $C_H(Q)$ as subgroups of $\hat N_G(Q_\delta)$.
Then it is easily checked that $k^*$ is central in $\hat N_G(Q_\delta)$, that $C_H(Q)$ is normal in $\hat N_G(Q_\delta)$ and that the intersection of $k^*$ and $C_H(Q)$ is trivial.
Since any
$k$-algebra automorphism on $U$ is inner, the quotient of $\hat N_G(Q_\delta)$ by $k^*$ is isomorphic to $N_G(Q_\delta)$. That is to say, $\hat N_G(Q_\delta)$ is a $k^*$-group with $k^*$-quotient $N_G(Q_\delta)$ (see paragraph \ref{Notation6}).
}
\end{order}

\begin{order}\label{Def-hat-L}{\rm
For an object $R_\varepsilon$ of ${\cal E}_{(P_\gamma,\,H,\,G)}$, we denote by $E_{G,\,\dot G}(R_\varepsilon)$ its automorphism group in ${\cal E}_{(P_\gamma,\,H,\,G)}$. The map $N_G(Q_\delta)\rightarrow E_{G,\,\dot G}(Q_\delta),\, x\mapsto (\varphi_{Q,\, x}^Q, \,\dot{x})$ induces a group isomorphism ${\cal E}\rightarrow E_{G,\,\dot G}(Q_\delta)$, through which we identify the two groups. We set $\hat E_{G,\,\dot G}(Q_\delta)=\hat N_G(Q_\delta)/C_H(Q)$. Then the inclusion $k^*\subset \hat N_G(Q_\delta)$ induces an injective group homomorphism $k^*\rightarrow \hat E_{G,\,\dot G}(Q_\delta)$ so that we can identify $k^*$ as a subgroup of $\hat E_{G,\,\dot G}(Q_\delta)$. Moreover $k^*$ is central in $\hat E_{G,\,\dot G}(Q_\delta)$ and the quotient of $\hat E_{G,\,\dot G}(Q_\delta)$ by $k^*$ is isomorphic to $E_{G,\,\dot G}(Q_\delta)$. That is to say, $\hat E_{G,\,\dot G}(Q_\delta)$ is a $k^*$-group with $k^*$-quotient $E_{G,\,\dot G}(Q_\delta)$. We denote by $\hat {\cal L}$ the pull-back through the homomorphism $\pi$ and the canonical homomorphism $\hat E_{G,\,\dot G}(Q_\delta)\rightarrow E_{G,\,\dot G}(Q_\delta)$. Then $\hat{\cal L}$ is a $k^*$-group with $k^*$-quotient ${\cal L}$.}
\end{order}

\begin{order}\label{Local} {\rm We define a $k^*$-group $\hat E_{{\cal N},\,{\cal E}}(Q_{\bbdelta})$ with $k^*$-quotient $E_{{\cal N},\,{\cal E}}(Q_{\bbdelta})$ as we define $\hat E_{G,\, \dot G}(Q_\delta)$ above. We set $\hat L={\rm Res}_\bbpi(\hat E_{{\cal N},\,{\cal E}}(Q_{\bbdelta}))$. Since the simple factors of $({\cal O} H)^Q$ and $({\cal O} {\cal H})^Q$ determined by the points $\delta$ and $\bbdelta$ respectively are isomorphic as $N_G(Q_\delta)$-algebras, it is easily checked that there is a $k^*$-group isomorphism $\hat E_{G,\, \dot G}(Q_\delta)\cong \hat E_{{\cal N},\,{\cal E}}(Q_{\bbdelta})$ lifting the equality $E_{G,\, \dot G}(Q_\delta)= E_{{\cal N},\,{\cal E}}(Q_{\bbdelta})$ and then that there is a $k^*$-group isomorphism $\hat{\cal L}\cong \hat L$ lifting the equality ${\cal L}=L$. We identify $\hat L$ and $\hat{\cal L}$ through the latter $k^*$-group isomorphism.}
\end{order}

\begin{order} {\rm Since $Q$ is normal in ${\cal L}$, the twisted group algebra ${\cal O}_*\hat{\cal L}^\circ$ has an obvious ${\cal L}/Q$-graded algebra structure. We identify ${\cal L}/Q$ and ${\cal E}$ through the isomorphism ${\cal L}/Q\cong {\cal E}$ induced by $\pi$. Then ${\cal O}_*\hat{\cal L}^\circ$ becomes an ${\cal E}$-graded algebra. Since ${\mathbbm b}={\rm Tr}^{\mathbbm G}_{{\cal N}}(b_\delta)$ and $P_{\bbgamma}\leq {\cal N}_\beta$, we have ${\mathbbm A}_\bbgamma=\mathbbm{i}{\mathbbm A}\mathbbm{i}=\mathbbm{i}({\cal O} {\cal N})\mathbbm{i}$. Since $P_\bbgamma$ is a pointed group on ${\cal O} {\cal H}$, ${\mathbbm A}_\bbgamma$ has an obvious ${\cal E}$-graded algebra structure.
By \cite[Corollary 3.15]{PZ2}, there is a $P$-interior full matrix algebra ${\mathbbm S}$ over ${\cal O}$ such that we have an ${\cal E}$-graded $P$-interior algebra isomorphism
\begin{equation}
{\mathbbm A}_\bbgamma\cong {\mathbbm S}\otimes_{\cal O} {\cal O}_*\hat{\cal L}^\circ.
\tag*{3.10.1}
\label{eq:3.10.1}
\end{equation}
Since $N_H(Q_\delta)$ is normal in $N_G(Q_\delta)$, ${\mathbbm A}_\bbgamma$ also has an obvious $N_G(Q_\delta)/N_H(Q_\delta)$-graded algebra structure.
By \cite[Theorem 1.2]{P1}, $H$ acts on the set of all defect pointed groups of $H_\alpha$. Clearly the $G$-conjugation action also stabilizes this set and thus by Frattini argument, we have $G=H N_G(Q_\delta)$. Then the inclusion $N_G(Q_\delta)\subset G$ induces a group isomorphism $\dot G\cong N_G(Q_\delta)/N_H(Q_\delta)$.
We  identify  $N_G(Q_\delta)/N_H(Q_\delta)$ with $\dot G$.
Thus ${\mathbbm A}_\bbgamma$ becomes a $\dot G$-graded algebra. Let $x$ be an element of $N_G(Q_\delta)$, $\tilde x$ be the image of $x$ in ${\cal E}$, $y_x$ be an inverse image of $\tilde x$ in ${\cal L}$ and $\hat y_x$ be a lift of $y_x$ in $\hat{\cal L}$.
We denote by ${\cal K}$ the inverse image in ${\cal L}$ of $N_H(Q_\delta)/C_H(Q)$ and let $\hat{\cal K}$ be the inverse image in $\hat{\cal L}$ of ${\cal K}$.
Then the isomorphism \ref{eq:3.10.1} maps $\mathbbm{i}({\cal O} N_H(Q_\delta)x)\mathbbm{i}$ onto ${\mathbbm S}\otimes_{\cal O} ({\cal O}_*\hat{\cal K}^\circ)\hat y_x$ and thus it is also a $\dot G$-graded $P$-interior algebra isomorphism.}
\end{order}

\section  { $p$-extensions of inertial blocks}

Throughout this section, we keep the notation in paragraph \ref{Notation12}. We assume that $P$ is abelian, that the index of $H$ in $G$ is a power of $p$, and that $b$ as a block of $G$ is inertial.
In this case, it is known that $P$ is a defect group of the block $b$ of $G$ and that $G$ is equal to $PH$.

\begin{order}\label{p-Cal-A-ga} {\rm Since the index of $H$ in $G$ is a power of $p$, by \cite[Example 3.9 and Theorem 3.16]{KP}, there exists a local point $\tilde\gamma$ of $P$ on ${\cal O} G$ such that $\gamma\subset \tilde\gamma$.
In particular, $P_{\tilde\gamma}$ is a defect pointed group of $G_{\{b\}}$ and ${\cal A}_\gamma$ is a source algebra of the block algebra ${\cal O} Gb$.
Using the simple factor of $({\cal O} G)^P$ determined by the point $\tilde{\gamma}$, in a way analogous to the definitions in paragraphs \ref{Def-hatNGQ} and \ref{Def-hat-L}, we can define a $k^*$-group $\hat{E}_G(P_{\tilde{\gamma}})$ whose $k^*$-quotient is $E_G(P_{\tilde\gamma})=N_G(P_{\tilde{\gamma}})/C_G(P)$.
Since $b$ as a block of $G$ is inertial,
by \cite{P11}, there is a $P$-interior full matrix algebra $S$ over ${\cal O}$ such that there is a $P$-interior algebra isomorphism
\begin{equation}
\mathcal{A}_\gamma \cong S \otimes_{\mathcal{O}} \mathcal{O}_*(P \rtimes  \hat{E}_G(P_{\tilde{\gamma}})^\circ);
\tag*{4.1.1}
\label{eq:4.1.1}
\end{equation}
moreover $S$ is
a primitive Dade $P$-algebra (see \cite[3.9]{P11}), unique up to isomorphisms (see \cite[Lemma 4.6]{P8}).
 The $P$-interior algebra $S$ is also $E_G(P_{\tilde\gamma})$-stable (see \cite[3.9]{P11}). We adjust the $P$-interior algebra structure on $S$ by an $E_G(P_{\tilde\gamma})$-stable linear character of $P$ so that $S$ is an $E_G(P_{\tilde\gamma})$-stable determinant one $P$-interior full matrix algebra. Then the $P$-interior algebra $S$ is unique up to $P$-interior algebra isomorphisms.}
\end{order}

\begin{order} {\rm Set $\hat{\cal A}=S^\circ\otimes_{\cal O} {\rm Res}^G_P({\cal A})$ and $\hat{{{\cal B}}}=S^\circ\otimes_{\cal O} {\rm Res}^G_P({{{\cal B}}})$. Then $\hat{\cal A}$ is a $\dot G$-graded $P$-interior algebra (see paragraphs \ref{Notation1} and \ref{Notation2}) and $\hat {{{\cal B}}}$ is a $Q$-interior $P$-algebra (see paragraph \ref{Notation9}).
By \cite[Theorem 5.3]{P4}, the local pointed group $P_\gamma$ determines a local pointed group $P_{\hat\gamma}$ on $\hat{{{\cal B}}}$ such that for some $\hat{i}\in \hat\gamma$, we have
\begin{equation}
\hat{i}(1\otimes i)=(1\otimes i)\hat{i}=\hat{i}.
\tag*{4.2.1}
\label{eq:4.2.1}
\end{equation}
Similarly $Q_\delta$ also determines a local pointed group $Q_{\hat\delta}$ on $\hat{{{\cal B}}}$ such that for some $\jmath$ in the unique local point of $Q$ on $S$ and some $\hat j$ in $\hat\delta$, we have
$\hat{j}(\jmath\otimes j)=(\jmath\otimes j)\hat{j}=\hat{j}.$
We set $\hat{\cal A}_{\hat\gamma}=\hat{i}\hat{\cal A}\hat{i}$, $\hat{{{\cal B}}}_{\hat\gamma}=\hat{i}\hat{{{\cal B}}}\hat{i}$, $\hat{\cal A}_{\hat\delta}=\hat{j}\hat{\cal A}\hat{j}$ and $\hat{{{\cal B}}}_{\hat\delta}=\hat{j}\hat{{{\cal B}}}\hat{j}$.
Then $\hat{\cal A}_{\hat\gamma}$ is a $\dot G$-graded $P$-interior algebra and $\hat{\cal A}_{\hat\delta}$ is a $\dot G$-graded $Q$-interior algebra (see paragraph \ref{Notation9}).
Since $\hat{\cal A}=\sum_u \hat{{{\cal B}}} u$ where $u$ run over $P$, by \cite[Proposition 2.8]{P5} $P_{\hat\gamma}$ determines a local pointed group $P_{\tilde{\hat\gamma}}$ on $\hat{\cal A}$
such that $\hat\gamma\subset \tilde{\hat\gamma}$. We note that the point $\tilde{\hat\gamma}$ of $P$ on $\hat{\cal A}$ is determined by the point $\tilde\gamma$ of $P$ on ${\cal A}$ in the sense of \cite[Theorem 5.3]{P4}.}
\end{order}

\begin{lem}\label{Iso-Hat-A-ga} With the notation as above, there is a $P$-interior algebra isomorphism
\begin{equation}
\hat{\cal A}_{\hat\gamma}\cong {\cal O}_*(P\rtimes \hat E_G(P_{\tilde\gamma})^\circ).
\tag*{4.3.1}
\label{eq:4.3.1}
\end{equation}
\end{lem}

\begin{proof} By isomorphism \ref{eq:4.1.1} and equation \ref{eq:4.2.1}, we get a $P$-interior algebra embedding
\begin{equation}
\hat{\cal A}_{\hat\gamma}\rightarrow S^\circ\otimes_{\cal O} S\otimes_{\cal O} {\cal O}_*(P\rtimes \hat E_G(P_{\tilde\gamma})^\circ).
\tag*{4.3.2}
\label{eq:4.3.2}
\end{equation}
There is a $P$-interior algebra embedding $d: {\cal O}\rightarrow S^\circ\otimes_{\cal O} S$, which induces another $P$-interior algebra embedding
\begin{equation}
{\cal O}_*(P\rtimes \hat E_G(P_{\tilde\gamma})^\circ)\rightarrow S^\circ\otimes_{\cal O} S\otimes_{\cal O}{\cal O}_*(P\rtimes \hat E_G(P_{\tilde\gamma})^\circ)
\tag*{4.3.3}
\label{eq:4.3.3}
\end{equation}
mapping $a$ onto $d(1)\otimes a$ for any $a\in {\cal O}_*(P\rtimes \hat E_G(P_{\tilde\gamma})^\circ)$.
Since we have $C_{P\rtimes E_G(P_{\tilde\gamma})}(P)=P$,  it is easy to check that $\{1\}$ is the unique local point of $P$ on ${\cal O}_*(P\rtimes \hat E_G(P_{\tilde\gamma})^\circ)$. Thus by \cite[Theorem 5.3]{P4}, $P$ has a unique local point on $S^\circ\otimes_{\cal O} S\otimes_{\cal O} {\cal O}_*(P\rtimes \hat E_G(P_{\tilde\gamma})^\circ)$. Therefore the local points of $P$ on $S^\circ\otimes_{\cal O} S\otimes_{\cal O} {\cal O}_*(P\rtimes \hat E_G(P_{\tilde\gamma})^\circ)$ (see paragraph \ref{Notation6}) determined respectively by the local pointed group $P_{\{1\}}$ on ${\cal O}_*(P\rtimes \hat E_G(P_{\tilde\gamma})^\circ)$ through embedding \ref{eq:4.3.3} and by the obvious local pointed group $P_{\{\hat i\}}$ on $\hat{\cal A}_{\hat\gamma}$ through embedding \ref{eq:4.3.2} have to coincide. Then by \cite[2.13.1]{P6}, the lemma is proved.
\end{proof}

\begin{lem}\label{Hat-i} With the notation as above, $\{\hat i\}$ is a local point of $Q$ on $\hat{{{\cal B}}}_{\hat\gamma}$ and thus $\hat i\in \hat\delta$.
\end{lem}

\begin{proof} Since we have $C_{P\rtimes E_G(P_{\tilde\gamma})}(Q)=P$ (see isomorphism \ref{eq:4.1.1}) and  $P\cap H=Q$, it is easy to check that $({\cal O}_*(P\rtimes \hat E_G(P_{\tilde\gamma})^\circ))(Q)$
is isomorphic to $kP$ as $k$-algebras. Since ${\rm Ker}({\rm Br}_Q^{{\cal O}_*(P\rtimes \hat E_G(P_{\tilde\gamma})^\circ)})$ is contained in $J({\cal O}_*(P\rtimes \hat E_G(P_{\tilde\gamma})^\circ))$,  $\{1\}$ is the unique local point of $Q$ on the twisted group algebra ${\cal O}_*(P\rtimes \hat E_G(P_{\tilde\gamma})^\circ)$. Then by Lemma \ref{Iso-Hat-A-ga}, $\{\hat{i}\}$ is a local point of $Q$ on $\hat{\cal A}_{\hat\gamma}$.
Since $\hat{\cal A}_{\hat\gamma}$ is equal to $\sum_{u\in P}\hat{{{\cal B}}}_{\hat\gamma}u$, by \cite[Proposition 2.8]{P5} any local point of $Q$ on $\hat{{{\cal B}}}_{\hat\gamma}$ is contained in a local point of $Q$ on $\hat{\cal A}_{\hat\gamma}$. This forces that $\{\hat i\}$ is a local point of $Q$ on $\hat{{{\cal B}}}_{\hat\gamma}$.
\end{proof}

So there is an invertible element $a$ in $\hat{{{\cal B}}}^Q$ such that $\hat j^a=\hat i$. In particular, the $a$-conjugation induces a $\dot G$-graded $Q$-interior algebra isomorphism $\hat{\cal A}_{\hat\delta}\cong {\rm Res}^P_Q(\hat {\cal A}_{\hat\gamma})$, through which we identify $\hat{\cal A}_{\hat\gamma}$ with $\hat{\cal A}_{\hat\delta}$. In this case, we have $\hat j=\hat i$. We denote by $N_{\hat{{{\cal B}}}_{\hat\gamma}^*}(P)$ the normalizer of $P\hat i$ in $\hat{{{\cal B}}}_{\hat\gamma}^*$.

\begin{order}{\rm
Analogous to the definition of $\hat{E}_G(P_{\tilde{\gamma}})$, we define a $k^*$-group $\hat{E}_H(Q_{\delta})$ whose $k^*$-quotient is given by $E_H(Q_{\delta}) = N_H(H_{\delta}) / C_H(Q)$.
Recall that there exists an isomorphism
\[
N_G(Q_\delta) / C_H(Q) \cong E_{G, \dot G}(Q_\delta),
\]
and that $\hat{E}_{G, \dot G}(Q_\delta)$ is a $k^*$-group with $k^*$-quotient $E_{G, \dot G}(Q_\delta)$ (see paragraph~\ref{Def-hat-L}).
Identifying $N_G(Q_\delta) / C_H(Q)$ with $E_{G, \dot G}(Q_\delta)$ via the isomorphism above, it follows that $\hat{E}_H(Q_{\delta})$ coincides with the inverse image of $N_H(Q_{\delta}) / C_H(Q)$ in $\hat{E}_{G, \dot G}(Q_\delta)$.
By the proof of Lemma \ref{Iso-Hat-A-ga}, $\gamma$ is the unique local point of $P$ on ${\cal O} H$ such that $\gamma\subset \tilde\gamma$, so we have$N_G(P_{\tilde\gamma})=N_G(P_\gamma)$. Since $G=PH$ and $Q=P\cap H$, we have $N_G(P_{\gamma})\subseteq N_G(Q_{\delta})$.
By Lemma \ref{N_G(R)}, we have $N_G(Q_\delta)=N_{N_G(Q_\delta)}(P_\gamma) C_G(Q_\delta)$.
So $N_G(Q_\delta)=N_{G}(P_\gamma) C_G(Q_\delta)$.
Then the inclusion map  $N_{G}(P_\gamma)\hookrightarrow N_G(Q_\delta)$ induced an isomorphism
 \begin{equation}
E_G(P_{\tilde{\gamma}})\cong E_H(Q_{\delta}).
\tag*{4.5.1}
\label{eq:4.5.1}
\end{equation}
}
\end{order}

\begin{lem}\label{k*-Iso} Keep the notation as above and denote by ${\frak e}_{\delta,\, \gamma}$ the inverse of isomorphism \ref{eq:4.5.1}. Then there is a $k^*$-group isomorphism $\hat {\frak e}_{\delta,\, \gamma}: \hat E_H(Q_\delta)\cong \hat E_G(P_{\tilde\gamma})$ lifting the group isomorphism ${\frak e}_{\delta,\, \gamma}$.
\end{lem}

\smallskip\noindent{\it Proof.}\quad Since the index of $H$ in $G$ is a $p$-power, the inclusion map ${\cal O} H\hookrightarrow {\cal O} G$ is a strict semicovering of $P$-algebra homomorphisms (see \cite[Example 3.9 and Theorem 3.16]{KP}). We denote by $\tilde\delta$ the local point of $Q$ on ${\cal O} G$ such that $\delta\subset \tilde\delta$. By the proof of Lemma \ref{Iso-Hat-A-ga}, $\delta$ is the unique local point of $Q$ on ${\cal O} H$ such that $\delta\subset \tilde\delta$, so we have $N_G(Q_\delta)=N_G(Q_{\tilde\delta})$ and then  $E_H(Q_\delta)=E_G(Q_{\tilde\delta})$. Then by \cite[Proposition 3.3]{KP}, the inclusion map ${\cal O} H\hookrightarrow {\cal O} G$ induces an $N_G(Q_\delta)$-algebra isomorphism $({\cal O} H)(Q_\delta)\cong({\cal O} G)(Q_{\tilde\delta})$ which induces a $k^*$-group isomorphism $\hat E_H(Q_\delta)\cong\hat E_G(Q_{\tilde\delta})$ lifting the equality $E_H(Q_\delta)=E_G(Q_{\tilde\delta})$. So it suffices to show that there is a $k^*$-group isomorphism $\hat {\frak e}_{\delta,\, \gamma}': \hat E_G(Q_{\tilde\delta})\cong \hat E_G(P_{\tilde\gamma})$ lifting the group isomorphism ${\frak e}_{\delta,\, \gamma}:E_G(Q_{\tilde\delta})\cong E_G(P_{\tilde\gamma})$. Since $b_\delta$ and $b_\gamma$ are nilpotent blocks of $C_H(Q)$ and $C_H(P)$ respectively and the index of $H$ in $G$ is a $p$-power, by \cite[Proposition 6.5]{KP} $Q_{\tilde\delta}$ and $P_{\tilde\gamma}$ are nilcentralized in the sense of \cite[3.1]{P5}. Since $Q_{\tilde\delta}$ is contained in $P_{\tilde\gamma}$,
the isomorphism $\hat {\frak e}_{\delta,\, \gamma}'$ follows from \cite[Theorem 7.16]{P10}.

\begin{order} {\rm We consider the twisted group algebra ${\cal O}_*(P\rtimes \hat E_G(P_{\tilde\gamma})^\circ)$ and begin to endow it with a $\dot{G}$-graded $P$-interior algebra structure.
Clearly the $k^*$-group isomorphism $\hat{\frak e}_{\delta,\, \gamma}$ can be extended to an injective group homomorphism
$Q\rtimes \hat E_H(Q_\delta)\rightarrow P\rtimes \hat E_G(P_{\tilde\gamma}),$ which extends the inclusion $Q\subset P$ and is still denoted by $\hat{\frak e}_{\delta,\, \gamma}$ for convenience. We identify $Q\rtimes \hat E_H(Q_\delta)$ with a subgroup of $P\rtimes \hat E_G(P_{\tilde\gamma})$ through $\hat {\frak e}_{\delta,\, \gamma}$, which is normal in $P\rtimes \hat E_G(P_{\tilde\gamma})$.}
\end{order}

\begin{order}{\rm Clearly $P\subset P\rtimes \hat E_G(P_{\tilde\gamma})$ and $P\subset G$ induce the following group isomorphisms $$(P\rtimes \hat E_G(P_{\tilde\gamma}))/(Q\rtimes \hat E_H(Q_\delta))\cong P/Q\cong \dot{G}. $$ We identify $\dot{G}$ with $(P\rtimes \hat E_G(P_{\tilde\gamma}))/(Q\rtimes \hat E_H(Q_\delta))$ through these isomorphisms. As $\cal A$ is endowed with a $\dot{G}$-graded algebra structure (see paragraph \ref{Notation12}), we endow the twisted
group algebra ${\cal O}_*(P\rtimes \hat E_G(P_{\tilde\gamma})^\circ)$ with an obvious $\dot{G}$-graded algebra structure. Notice that the $1$-component of this $\dot{G}$-graded algebra ${\cal O}_*(P\rtimes \hat E_G(P_{\tilde\gamma})^\circ)$ is ${\cal O}_*(Q\rtimes \hat E_H(Q_\delta)^\circ)$. The inclusion $P\subset P\rtimes \hat E_G(P_{\tilde\gamma})$ induces a $P$-interior algebra structure on ${\cal O}_*(P\rtimes \hat E_G(P_{\tilde\gamma})^\circ)$.
}
\end{order}

\begin{prop}\label{Graded-Hat-A-ga} With the notation as above,
there is a $\dot{G}$-graded $P$-interior algebra isomorphism
$\hat{\cal A}_{\hat\gamma}\cong {\cal O}_*(P\rtimes \hat E_G(P_{\tilde\gamma})^\circ).$
\end{prop}

\begin{proof} Since $\hat{\cal A}$ is a $\dot G$-graded $P$-interior algebra with the 1-component $\hat{{{\cal B}}}$, the group $F_{\hat{{{\cal B}}}}(P_{\hat\gamma})$ (see paragraph \ref{Notation11}) makes sense. In this case, $F_{\hat{{{\cal B}}}}(P_{\hat\gamma})$ actually is equal to the group obtained by applying $F_A(K_\gamma)$ in \cite[2.3]{P5} to $\hat {{{\cal B}}}$ and $P_{\hat\gamma}$. By \cite[Proposition 2.8]{P5} applied to the inclusion map $\hat{{{\cal B}}}\rightarrow \hat{\cal A}$, $F_{\hat{{{\cal B}}}}(P_{\hat\gamma})$ is a normal subgroup of $F_{\hat{\cal A}}(P_{\tilde{\hat\gamma}})$ with a $p$-power index. On the other hand, by \cite[Theorem 3.1]{P2}, we have $E_G(P_{\tilde\gamma})= F_{{\cal A}}(P_{\tilde\gamma})$ and by \cite[Lemma 1.17]{KP}, we have $F_{\cal A}(P_{\tilde\gamma})= F_{\hat{\cal A}}(P_{\tilde{\hat\gamma}})$. Since $E_G(P_{\tilde{\gamma}})$ is a $p'$-group, so is $F_{\hat{\cal A}}(P_{\tilde{\hat\gamma}})$ and thus we have $F_{\hat{{{\cal B}}}}(P_{\hat\gamma})=F_{\hat{\cal A}}(P_{\tilde{\hat\gamma}})$.

For any $\phi\in F_{\hat{{{\cal B}}}}(P_{\hat\gamma})$, there is an invertible element $a$ in $N_{\hat{{{\cal B}}}_{\hat\gamma}^*}(P)$ such that $\phi(u)\hat i=aua^{-1}$ for any $u\in P$. Since the map $P\rightarrow P\hat i,\,u\mapsto u\hat i$ is a group isomorphism, it is easy to check that the map $F_{\hat{{{\cal B}}}}(P_{\hat\gamma})\rightarrow N_{\hat{{{\cal B}}}_{\hat\gamma}^*}(P)/(\hat{{{\cal B}}}_{\hat\gamma}^P)^*$ mapping $\phi$ onto the image of $a$ in $N_{\hat{{{\cal B}}}_{\hat\gamma}^*}(P)/(\hat{{{\cal B}}}_{\hat\gamma}^P)^*$ for any $\phi\in F_{\hat{{{\cal B}}}}(P_{\hat\gamma})$ is a group isomorphism. Then we have a short exact sequence of group homomorphisms $$1\rightarrow (1+J(\hat{{{\cal B}}}_{\hat\gamma}^P))\cong (\hat{{{\cal B}}}_{\hat\gamma}^P)^*/k^*\rightarrow N_{\hat{{{\cal B}}}_{\hat\gamma}^*}(P)/k^*\rightarrow F_{\hat{{{\cal B}}}}(P_{\hat\gamma})\rightarrow 1. $$
Since $F_{\hat{{{\cal B}}}}(P_{\hat\gamma})$ is a $p'$-group, this short exact sequence uniquely splits (see \cite[Lemma 45.6]{T}) and thus there is a subgroup $\hat F$ of $N_{\hat{{{\cal B}}}_{\hat\gamma}^*}(P)$ containing $k^*$ such that the restriction to $\hat F/k^*$ of the homomorphism $N_{\hat{{{\cal B}}}_{\hat\gamma}^*}(P)/k^*\rightarrow F_{\hat{{{\cal B}}}}(P_{\hat\gamma})$ is a group isomorphism.

We consider the conjugation action of $\hat F$ on $P\hat i$, the semidirect product $(P\hat{i})\rtimes \hat F$ and the ${\cal O}$-algebra homomorphism $\theta: {\cal O}_*((P\hat{i})\rtimes \hat F)\rightarrow \hat{\cal A}_{\hat\gamma}$ induced by the inclusions $P\hat{i}\subset \hat{\cal A}_{\hat\gamma}$ and $\hat F\subset \hat{\cal A}_{\hat\gamma}$. Similar to the paragraph above, we have another uniquely split short exact sequence
$$1\rightarrow (1+J(\hat{\cal A}_{\hat\gamma}^P))\rightarrow N_{\hat{\cal A}_{\hat\gamma}^*}(P)/k^*\rightarrow F_{\hat{\cal A}}(P_{\tilde{\hat\gamma}})\rightarrow 1. $$
Clearly the inclusion $\hat F\subset N_{\hat{{{\cal B}}}_{\hat\gamma}^*}(P)$ induces a group isomorphism $F_{\hat{\cal A}}(P_{\tilde{\hat\gamma}})\cong \hat F/k^*$, which is a section of the homomorphism
$N_{\hat{\cal A}_{\hat\gamma}^*}(P)/k^*\rightarrow F_{\hat{\cal A}}(P_{\tilde{\hat\gamma}})$. Let $\hat L$ be the inverse image of $\hat E_G(P_{\tilde\gamma})^\circ$ through isomorphism \ref{eq:4.3.1}. The inclusion $\hat L\subset N_{\hat{\cal A}_{\hat\gamma}^*}(P)$ induces a group isomorphism $F_{\hat{\cal A}}(P_{\tilde{\hat\gamma}})\cong \hat L/k^*$, which is also a section of the homomorphism
$N_{\hat{\cal A}_{\hat\gamma}^*}(P)/k^*\rightarrow F_{\hat{\cal A}}(P_{\tilde{\hat\gamma}})$.
So $\hat F$ and $\hat L$
are conjugate in $\hat{\cal A}_{\hat\gamma}^*$. Thus the homomorphism $\theta$ is surjective and an ${\cal O}$-algebra isomorphism. Note that $\theta$ maps ${\cal O}_*((Q\hat{i})\rtimes \hat F)$ onto $ \hat{{{\cal B}}}_{\hat\gamma}$.

Clearly the inclusions $\hat F\subset N_{\hat{{{\cal B}}}_{\hat\gamma}^*}(Q)$ and $\hat F\subset N_{\hat{{{\cal B}}}_{\hat\gamma}^*}(P)$ induce $k^*$-group isomorphisms ${\frak f}_\delta: \hat F\cong \hat F_{\hat{{{\cal B}}}}(Q_{\hat\delta})$ and ${\frak f}_\gamma: \hat F\cong \hat F_{\hat{\cal A}}(P_{\tilde{\hat\gamma}})$, respectively. On the other hand, the obvious inclusion $N_{\hat{{{\cal B}}}_{\hat\gamma}^*}(P)\subset N_{\hat{{{\cal B}}}_{\hat\gamma}^*}(Q)$ induces a group isomorphism $F_{\hat{\cal A}}(P_{\tilde{\hat\gamma}})\cong F_{\hat{{{\cal B}}}}(Q_{\hat\delta})$ which maps any automorphism on $P$ onto its restriction to $Q$. Since it follows from \cite[Theorem 3.1]{P2} and \cite[Lemma 1.17]{KP} that $E_H(Q_\delta)$ and $E_G(P_{\tilde\gamma})$ are equal to $F_{\hat{{{\cal B}}}}(Q_{\hat\delta})$ and $F_{\hat{\cal A}}(P_{\tilde{\hat\gamma}})$ respectively, the group isomorphism $F_{\hat{\cal A}}(P_{\tilde{\hat\gamma}})\cong F_{\hat{{{\cal B}}}}(Q_{\hat\delta})$ coincides with isomorphism \ref{eq:4.5.1}. The inclusion
$N_{\hat{{{\cal B}}}_{\hat\gamma}^*}(P)\subset N_{\hat{{{\cal B}}}_{\hat\gamma}^*}(Q)$ also induces a $k^*$-group isomorphism $\hat{\frak f}_{\gamma, \,\delta}:\hat F_{\hat{\cal A}}(P_{\tilde{\hat\gamma}})\cong \hat F_{\hat{{{\cal B}}}}(Q_{\hat\delta})$ lifting isomorphism \ref{eq:4.5.1}.
Moreover we have $\hat{\frak f}_{\gamma, \,\delta}\circ{\frak f}_\gamma={\frak f}_\delta$. By \cite[Proposition 6.12]{P6}, \cite[2.12.4]{KP} and \cite[Proposition 5.11]{P4}, we have $k^*$-group isomorphisms ${\frak g}_\gamma: \hat E_G(P_{\tilde\gamma})^\circ\cong \hat F_{\hat{\cal A}}(P_{\tilde{\hat\gamma}})$ and ${\frak g}_\delta:\hat E_H(Q_\delta)^\circ\cong \hat F_{\hat{{{\cal B}}}}(Q_{\hat\delta})$, which respectively lift the equalities $E_G(P_{\tilde\gamma})= F_{\hat{\cal A}}(P_{\tilde{\hat\gamma}})$
and $E_H(Q_\delta)=F_{\hat{{{\cal B}}}}(Q_{\hat\delta})$. We set $\hat{\frak e}'_{\delta,\, \gamma}={\frak g}_\gamma^{-1}\circ\hat{\frak f}_{\gamma, \,\delta}^{-1}\circ{\frak g}_\delta$. Since $\hat{\frak e}'_{\delta,\, \gamma}$ and $\hat{\frak e}_{\delta,\, \gamma}$ both lift the isomorphism ${\frak e}_{\delta,\, \gamma}$, we can adjust $\hat{\frak e}_{\delta,\, \gamma}$ by a linear character of $E_G(P_{\tilde\gamma})$ so that $\hat{\frak e}'_{\delta,\, \gamma}$ and $\hat{\frak e}_{\delta,\, \gamma}$ coincide. Summarizing the above, we have a commutative diagram of $k^*$-group homomorphisms \[
\begin{matrix}
\hat{E}_G(P_{\tilde{\gamma}})^\circ & \xrightarrow{{\mathfrak{g}}_\gamma} & \hat{F}_{\hat{\mathcal{A}}}(P_{\tilde{\hat{\gamma}}}) & \xleftarrow{{\mathfrak{f}}_\gamma} & \hat{F} \\
\hat{\mathfrak{e}}_{\delta, \gamma} \Big\uparrow & & \Big\downarrow \hat{\mathfrak{f}}_{\gamma, \delta} & & =\Big\downarrow \\
\hat{E}_H(Q_\delta)^\circ & \xrightarrow{{\mathfrak{g}}_\delta} & \hat{F}_{\hat{\mathcal{B}}}(Q_{\hat{\delta}}) & \xleftarrow{{\mathfrak{f}}_\delta} & \hat{F}
\end{matrix}.
\]
In particular, this shows that there is an ${\cal O}$-algebra isomorphism ${\cal O}_*((P\hat{i})\rtimes \hat F)\cong {\cal O}_*(P\rtimes \hat E_G(P_{\tilde\gamma})^\circ)$ mapping $u\hat i$ onto $u$ for any $u\in P$ and ${\cal O}_*((Q\hat{i})\rtimes \hat F)$ onto ${\cal O}_*(Q\rtimes \hat E_H(Q_\delta)^\circ)$. By composing the inverse of this ${\cal O}$-algebra isomorphism and isomorphism \ref{eq:4.3.1}, we get the desired $\dot G$-graded $P$-interior algebra isomorphism.
\end{proof}

\begin{lem}\label{Cal-A-ga} With the notation as above, there is a $\dot{G}$-graded $P$-interior algebra embedding ${\cal A}_\gamma\rightarrow S\otimes_{\cal O} \hat {\cal A}_{\hat\gamma}$.
\end{lem}

\begin{proof} Clearly the inclusion map $\hat {\cal A}_{\hat\gamma}\rightarrow S^\circ\otimes_{\cal O} {\cal A}_\gamma$ (see equation \ref{eq:4.2.1}) is a $\dot{G}$-graded $P$-interior algebra embedding and so is the inclusion map $\iota:S\otimes_{\cal O}\hat {\cal A}_{\hat\gamma}\rightarrow S\otimes_{\cal O} S^\circ\otimes_{\cal O} {\cal A}_\gamma.$ By \cite[Theorem 5.3]{P4}, the obvious local point $\{\hat i\}$ of $P$ on $\hat{{{\cal B}}}_{\hat\gamma}$ determines a local point $\hat\gamma'$ of $P$ on $S\otimes_{\cal O} \hat{{{\cal B}}}_{\hat\gamma}$.
Since embeddings preserve localness (see \cite[Proposition 15.1]{T}), there is a local point $\gamma'$ of $P$ on $S\otimes_{\cal O} S^\circ\otimes_{\cal O} {{{\cal B}}}_\gamma$ containing $\hat\gamma'$.

There is a $P$-interior algebra embedding $e: {\cal O}\rightarrow S\otimes_{\cal O} S^\circ$, which induces a $\dot{G}$-graded $P$-interior algebra embedding $e\otimes 1:{\cal A}_\gamma\rightarrow S\otimes_{\cal O} S^\circ\otimes_{\cal O}{\cal A}_\gamma$ mapping $a$ onto $a\otimes d(1)$ for any $a\in {\cal A}_\gamma$.
Obviously $\{i\}$ is a local point of $P$ on ${{{\cal B}}}_\gamma$.
Again since embeddings preserve localness, there is a local point $\gamma''$ of $P$ on $S\otimes_{\cal O} S^\circ\otimes_{\cal O} {{{\cal B}}}_\gamma$ containing the image of $h$ through embedding $e\otimes 1$.

Obviously $\{i\}$ is also the unique local point of $P$ on ${{{\cal B}}}_\gamma$ and then by \cite[Theorem 5.3]{P4}, $P$ has a unique local point on $S\otimes_{\cal O} S^\circ\otimes_{\cal O} {{{\cal B}}}_\gamma$. So $\gamma'$ is equal to $\gamma''$.
Then by \cite[2.11.2]{KP}, the embedding $e\otimes 1$ factors through the embedding $\iota$ to get the embedding in this lemma.
\end{proof}

\begin{thm}\label{Gad-Cal-A-ga} With the notation as above, there is a $\dot{G}$-graded $P$-interior algebra isomorphism
${\cal A}_\gamma\cong S\otimes_{\cal O} {\cal O}_*(P\rtimes \hat E_G(P_{\tilde\gamma})^\circ)$.
\end{thm}

\begin{proof} By  Proposition \ref{Graded-Hat-A-ga} and Lemma \ref{Cal-A-ga}, there is a $\dot{G}$-graded $P$-interior algebra embedding
${\cal A}_\gamma\rightarrow S\otimes_{\cal O} {\cal O}_*(P\rtimes \hat E_G(P_{\tilde\gamma})^\circ)$.  By Lemma \ref{k*-Iso}, it is easy to check that the identity element of $S\otimes_{\cal O} {\cal O}_*(P\rtimes \hat E_G(P_{\tilde\gamma})^\circ)$ is contained in the unique local point of $P$ on $S\otimes_{\cal O} {\cal O}_*(Q\rtimes \hat E_H(Q_\delta)^\circ)$. Thus this embedding must be an isomorphism.
\end{proof}

\section  {Proof of the main results}

We continue to keep the notation in paragraph \ref{Notation12} in this section. In addition, we assume that $P$ is abelian and that the block $b$ of $H$ is inertial. In this section, we give a proof of Theorem \ref{Main1} and Corollary \ref{Main2}.

\begin{order}\label{NotationF5}{\rm
Set $J=P H$, $\dot{J}=J/H$ and ${\cal C}_\gamma=i({\cal O} J)i$. Clearly ${\cal C}_\gamma$ is a $\dot J$-graded $P$-interior algebra. By \cite[Proposition 6.2]{KP}, the obvious $P$-algebra homomorphism ${\cal O} H\rightarrow {\cal O} J$ induced by the inclusion $H\subset J$ is a strict semicovering and there is a unique local point $\tilde\gamma$ of $P$ on ${\cal O} J$ containing $\gamma$. Since $P_\gamma$ is a defect pointed group of $J_\alpha$, by \cite[Corollary 6.3]{KP} $P_{\tilde\gamma}$ is a defect pointed group of $J_\alpha$ on ${\cal O} J$. In particular, ${\cal C}_\gamma$ is a source algebra of the block algebra ${\cal O} J b$.
By Theorem \ref{Gad-Cal-A-ga}, we get a determinant one $P$-interior full matrix algebra $S$ over ${\cal O}$ such that there is a $\dot{J}$-graded $P$-interior algebra isomorphism
\begin{equation}
{\cal C}_\gamma\cong S\otimes_{\cal O} {\cal O}_*(P\rtimes \hat E_J(P_{\tilde\gamma})^\circ).
\tag*{5.1.1}
\label{eq:5.1.1}
\end{equation}
Moreover $S$ is unique up to $P$-interior algebra isomorphisms and it has a $P$-stable ${\cal O}$-basis containing the unity of $S$.}
\end{order}

\begin{order}{\rm  Let $R_\varepsilon$ be a local pointed group
such that $Q_\delta\leq R_\varepsilon\leq P_\gamma$. Take some $h\in
\varepsilon$ such that $ih=hi=h$ and $hj=jh=j$. Set
${\cal A}_\varepsilon=h{\cal A}h$ and ${\cal
B}_\varepsilon=h{{{\cal B}}}h$. Clearly ${\cal
A}_\varepsilon=h{\cal A}h$ is a $\dot{G}$-graded $R$-interior algebra. Set $\hat{\cal A}=S^\circ\otimes_{\cal O} {\rm Res}^G_P(
{\cal A})$ and $\hat{{{\cal B}}}=S^\circ\otimes_{\cal O} {\rm Res}^G_P(
{{{\cal B}}}).$ Then $\hat{\cal A}$ is a $\dot G$-graded $P$-interior algebra.
By \cite[Theorem 5.3]{P4}, $R_\varepsilon$ determines
a unique local pointed group $R_{\hat\varepsilon}$ on $\hat{{{\cal B}}}$ such
that for some $\hat{h}\in \hat\varepsilon$ and some element $\ell$ of
the unique local point $\varepsilon_S$ of $R$ on $S$, we have $\hat{h}(\ell\otimes
h)=(\ell \otimes h)\hat{h}=\hat{h}.$ Set ${\hat{\cal A}}_{\hat\varepsilon}=\hat{h}{\hat{\cal A}}\hat{h}$ and ${\hat{{{\cal B}}}}_{\hat\varepsilon}=\hat{h}{\hat{{{\cal B}}}}\hat{h}$. Clearly ${\hat{\cal A}}_{\hat\varepsilon}$ is a $\dot{G}$-graded $R$-interior algebra.}
\end{order}

\begin{order} {\rm Similarly we have the local pointed groups $P_{\hat\gamma}$ and $Q_{\hat\delta}$ on $\hat {{{\cal B}}}$ determined by the pointed groups $P_\gamma$ and $Q_\delta$ on ${{{\cal B}}}$ respectively, the $\dot{G}$-graded $P$-interior algebra $\hat{\cal A}_{\hat\gamma}=\hat{i}\hat{\cal A}\hat{i}$ for some $\hat i\in \hat\gamma$, and the $\dot{G}$-graded $Q$-interior algebra $\hat{\cal A}_{\hat\delta}=\hat{j}\hat{\cal A}\hat{j}$ for some $\hat j\in \hat\delta$.
By Lemma \ref{Hat-i}, we have $\hat\gamma\subset \hat\varepsilon\subset
\hat\delta$. So as we did in the paragraph above Lemma \ref{k*-Iso}, we identify ${\hat{\cal A}}_{\hat\gamma}$ and ${\hat{\cal A}}_{\hat\varepsilon}$ through some suitable $\dot G$-graded $R$-interior algebra isomorphism, and ${\hat{\cal A}}_{\hat\varepsilon}$ and ${\hat{\cal A}}_{\hat\delta}$ through some suitable $\dot G$-graded $Q$-interior algebra isomorphism. In this sense, we have ${\hat{\cal A}}_{\hat\gamma}={\hat{\cal A}}_{\hat\varepsilon}={\hat{\cal A}}_{\hat\delta}$ and $\hat i=\hat h=\hat j$.}
\end{order}

\begin{lem}\label{Inj-map} With the above notation, we have $E_{G,\,\dot{G}}(R_\varepsilon)\subset {F}_{{\cal A},\,\dot{G}}(R_{\varepsilon})$ and $E_{G,\,\dot{G}}(R_\varepsilon)\subset {F}_{\hat{\cal A},\,\dot{G}}(R_{\hat\varepsilon})$.
\end{lem}

\begin{proof} Given $(\phi,\, \dot x)$ in ${ E}_{G, \, \dot G}(R_\varepsilon)$, $\dot x$ has a representative $y$ in $N_G(R_\varepsilon)$ such that we have $\phi=\varphi_{R,\, y}^R$. There is some invertible element $a_y$ of
$({\cal O} H)^R$ such that $yhy^{-1}=a_y ha_y^{-1}$. Thus
$a_y^{-1}y$ commutes with $h$. Set $d_y=(a_y^{-1}y)h$. Clearly $d_y$ belongs to ${\cal N}_{{\cal A}_\varepsilon^*}^{
\dot{x}}(R)$ and $d_yud_y^{-1}=\phi(u)h$. So $(\phi,\, \dot x)\in {F}_{{\cal A},\,\dot{G}}(R_{\varepsilon})$ and $E_{G,\,\dot{G}}(R_\varepsilon)\subset {F}_{{\cal A},\,\dot{G}}(R_{\varepsilon})$.

Set $T=\ell S \ell$. Clearly $T$ is a Dade $R$-algebra and a $R$-interior algebra. Since the $\dot J$-graded $P$-interior algebra ${\cal C}_\gamma$ is $N_G(P_\gamma)$-stable, by the uniqueness of $S$ the $P$-interior algebra $S$ is $N_G(P_\gamma)$-stable. Since $N_G(R_\varepsilon)=N_{N_G(R_\varepsilon)}(P_\gamma)C_G(R_\varepsilon)$ (see Lemma \ref{N_G(R)}), by the uniqueness of the local point of $R$ on $S$, $T$ is $N_G(R_\varepsilon)$-stable.
So there is an invertible element $s_y$ in $T$ such that $s_y u s_y^{-1}=\phi(u)1_{T}$.

Clearly we have $\hat{h}^{(s_y\otimes d_y)^{-1}}(1_T\otimes
h)=(1_T \otimes h)\hat{h}^{(s_y\otimes d_y)^{-1}}=\hat{h}^{(s_y\otimes d_y)^{-1}}$ and $\hat{h}^{(s_y\otimes d_y)^{-1}}$ is contained in some local point of $R$ on $\hat {{{\cal B}}}$. By the uniqueness of $R_{\hat\varepsilon}$, $\hat{h}^{(s_y\otimes d_y)^{-1}}$ belongs to $\hat\varepsilon$. In particular, there is some invertible element $e_y$ of $\hat{{{\cal B}}}^R$
such that $(s_y\otimes d_y)\hat{h} (s_y\otimes d_y)^{-1}= e_y \hat{h}
e_y^{-1}$. Set $c_y=e_y^{-1}(s_y\otimes d_y)\hat{h}$. Then $c_y$ is an
invertible element of the $\dot{x}$-component of $\hat{\cal A}_{\hat\varepsilon}$ and we have
$c_y u c_y^{-1}=\phi(u)\hat h$ for any $u\in R$. Thus we have $(\phi,\,\dot x)\in F_{\hat{\cal A},\,\dot G}(R_{\hat\varepsilon})$ and $E_{G,\,\dot{G}}(R_\varepsilon)\subset F_{\hat{\cal A},\,\dot G}(R_{\hat\varepsilon})$.
\end{proof}

\begin{order}{\rm The isomorphism $\tau$ of categories in Proposition \ref{Cal-L} induces a group isomorphism $\chi_R: E_{G,\,\dot G}(R_\varepsilon)\cong E_{{\cal L},\, {\cal E}}(R),$ where $E_{{\cal L},\, {\cal E}}(R)$ is the automorphism group of $R_{\{1\}}$ in ${\cal E}_{(P_{\{1\}},\,Q,\,{\cal L})}$. We also have group isomorphisms $\theta_R: N_{\cal L}(R)/Q\cong E_{{\cal L},\, {\cal E}}(R)$ and $\vartheta_R: {F}_{\hat{\cal A},\,\dot{G}}(R_{\hat\varepsilon})\cong {\cal F}_{\hat{\cal A}}(R_{\hat\varepsilon})$ (see paragraphs \ref{Def-hat-L} and \hyperlink{prop2111}{2.9.1}. respectively).  We denote by $\lambda_R$ the composition of $\theta_R$, $\chi_R^{-1}$, the inclusion map $E_{G,\,\dot{G}}(R_\varepsilon)\subset {F}_{\hat{\cal A},\,\dot{G}}(R_{\hat\varepsilon})$ and
$\vartheta_R$.}
\end{order}

\begin{order}{\rm The inclusion $N_G(R_\varepsilon)\subset N_G(Q_\delta)$ induces a group homomorphism $\theta:E_{G,\,\dot G}(R_\varepsilon)\rightarrow E_{G,\,\dot G}(Q_\delta)$ mapping $(\varphi^R_{R, \,x^{-1}},\,\dot x)$ onto $(\varphi^Q_{Q, \,x^{-1}},\,\dot x)$ for any $x\in N_G(R_\varepsilon)$. Similarly the inclusion $N_{\cal L}(R)\subset {\cal L}$ induces a group homomorphism $\vartheta: E_{{\cal L},\,{\cal E}}(R)\rightarrow E_{L,\,{\cal E}}(Q)$. We have $\vartheta\circ\chi_R=\chi_Q\circ \theta$ and then the following commutative diagram of group homomorphisms
\begin{equation}
\begin{matrix}
N_{\mathcal{L}}(R)/Q & \xrightarrow{\lambda_R} & \mathcal{F}_{\hat{\mathcal{A}}}(R_{\hat{\varepsilon}}) \\
\Big\downarrow & & \Big\downarrow \\
\mathcal{L}/Q & \xrightarrow{\lambda_Q} & \mathcal{F}_{\hat{\mathcal{A}}}(Q_{\hat{\delta}}),
\end{matrix}
\tag*{5.6.1}
\label{eq:dia-5.6.1}
\end{equation}
where the left downarrow is induced by the inclusion $N_{\cal L}(R)\subset {\cal L}$ and the right downarrow is induced by the inclusion ${\cal N}_{\hat{\cal A}_{\hat\varepsilon}^*}(R)\subset {\cal N}_{\hat{\cal A}_{\hat\delta}^*}(Q)$}
\end{order}

\begin{order}{\rm The canonical homomorphism
${\cal N}_{\hat{{\cal A}}_{\hat\delta}^*}(Q)\rightarrow {\cal N}_{\hat{{\cal A}}_{\hat\delta}^*}(Q)/k^*$ maps $(\hat{{{\cal B}}}_{\hat\delta}^Q)^*$ onto $(\hat{{{\cal B}}}_{\hat\delta}^Q)^*/k^*$ and induces a canonical group isomorphism
\begin{equation}
\Big({\cal N}_{\hat{{\cal A}}_{\hat\delta}^*}(Q)/k^*\Big)\Big/\Big((\hat{{{\cal B}}}_{\hat\delta}^Q)^*/k^*\Big)\cong {\cal N}_{\hat{{\cal A}}_{\hat\delta}^*}(Q)/(\hat{{{\cal B}}}_{\hat\delta}^Q)^*.
\tag*{5.7.1}
\label{eq:5.7.1}
\end{equation}
Set ${\frak M}={\cal N}_{\hat{{\cal A}}_{\hat\delta}^*}(Q)/k^*$ and ${\frak K}=(\hat{{{\cal B}}}_{\hat\delta}^Q)^*/k^*$. By composing the canonical surjective homomorphism ${\cal L}\rightarrow {\cal L}/Q$, the homomorphism $\lambda_Q$ and the inverse of this canonical group isomorphism, we get a group homomorphism ${\frak p}: {\cal L}\rightarrow {\frak M}/{\frak K}.$
There is an injective group homomorphism $\frak{i}: P\rightarrow \frak M$ mapping $u$ onto the image of $u\hat i$ in $\frak M$ for any $u\in P$.}
\end{order}

\begin{lem}\label{Lift-Cal-L-FrM} With the notation as above, the homomorphism ${\frak p}$ can be lifted to an injective group homomorphism
${\cal L}\rightarrow \frak M$, which extends the homomorphism $\frak i$.
\end{lem}

Before proving this lemma, we first prove the following result.

\begin{lem}\label{L-M} Let $L$ be a finite group, $M$
a group, $Z$ a normal subgroup of $M$~and
$\bar\sigma\,\colon L\to \bar M = M/Z$ a group homomorphism. Let $\{Z_n\}_{n\in \Bbb N\cup \{0\}}$ be a sequence of subgroups of $Z$, which are normal in $M$.
Assume that $Z_0=Z$, that $Z_n/Z_{n+1}$ is a $p'$-divisible abelian group for any $n\in \Bbb N\cup \{0\}$, and that $\lim\limits_{\leftarrow} Z/Z_n\cong Z$. Assume
that, for a Sylow $p$-subgroup $P$ of $L\,,$ there exists a
group homomorphism $\varsigma\,\colon P\to M$ lifting the restriction of
$\bar\sigma$ to $P$ and fulfilling the following condition:

\noindent {\bf 5.9.1}\hypertarget{cond:4.13.1}{} For any subgroup $R$ of $P$ and any $x\in L$ such that $R^x\subset P\,,$ then there is $y_x\in M$ such that $\bar\sigma (x) = \bar y_x$, where $\bar y_x$ denotes the image of $y_x$ in $\bar M$, and such that $\varsigma (u^x) = \varsigma (u)^{y_x}$ for any $u\in R$.

\vspace{0.5\baselineskip}
\noindent {\bf 5.9.2}\hypertarget{cond:4.13.2}{} $\varsigma(P)$ acts trivially on $Z$ by the conjugation.

\noindent Then there is a group homomorphism $\sigma\,\colon L\to
M$ lifting $\bar\sigma$ and extending~$\varsigma$. Moreover, if
$\sigma'\,\colon L\to M$ is another group homomorphism lifting
$\bar\sigma$ and extending $\varsigma$, there is $z\in Z$ such that
$\sigma' (x) = \sigma (x)^z$ for any $x\in L$.
\end{lem}

\begin{proof} For any $n\in \Bbb N\cup \{0\}$, we set $M_n=M/Z_n$; we denote by $\varsigma_n$ the group homomorphism $P\rightarrow M_n$ induced by $\varsigma$, and by $y_{x, n}$ the image in $M_n$ of $y_x$ in condition \hyperlink{cond:5.9.1}{5.9.1}; there is a canonical group isomorphism $M_n\cong M_{n+1}/(Z_n/Z_{n+1});$ we identify $M_n$ with $M_{n+1}/(Z_n/Z_{n+1})$ through this isomorphism. Clearly $M_0=\bar M$ and $y_{x,\, 0}=\bar y_x$.

Set $\sigma_0=\bar \sigma$. Clearly $\sigma_0$ extends $\varsigma_0$. Assume that there is a sequence of group homomorphisms $\{\sigma_l: L\rightarrow M_l\}_{0\leq l\leq k}$ such that for any integer $l$ such that $0\leq l\leq k-1$, we have $\sigma_{l+1}$ lifts $\sigma_{l}$, and such that for any integer $l$ such that $0\leq l\leq k$, $\sigma_l$ extends $\varsigma_l$. Let $R$ be a subgroup of $P$ and take
$x\in L$ such that $R^x\subset P$. Since $\sigma_k$ lifts $\sigma_0$, $y_{x, \,k+1}$ and any representative $z_{x, \, k+1}$ of $\sigma_k(x)$ in $M_{k+1}$ differentiate by some element of $Z/Z_{k+1}$. Moreover since $\varsigma(P)$ acts trivially on $Z$, we have $\varsigma_{k+1}(u^x)=\varsigma_{k+1}(u)^{y_{x, \, k+1}}=\varsigma_{k+1}(u)^{z_{x, \, k+1}}$ for any $u\in P$. Thus $\sigma_k$, $\varsigma_{k+1}$ and $z_{x, \,k+1}$ satisfy condition  \hyperlink{cond:5.9.1}{5.9.1}. Since $Z_k/Z_{k+1}$ is a $p'$-divisible abelian group, by \cite[Lemma 3.6]{PZ2} there is a group homomorphism $\sigma_{k+1}\,\colon L\to M_{k+1}$ lifting $\sigma_k$ and extending~$\varsigma_{k+1}$. By induction, we get a sequence of group homomorphisms $\{\sigma_n: L\rightarrow M_n\}_{n\in \Bbb N\cup \{0\}}$ such that $\sigma_n$ extends $\varsigma_n$ for any $n\in \Bbb N\cup \{0\}$ and such that $\sigma_{n+1}$ lifts $\sigma_n$ for any $n\in \Bbb N\cup \{0\}$.

Since $\lim\limits_{\leftarrow} Z/Z_n\cong Z$, by \cite[Lemma 45.4]{T} canonical homomorphisms $M\rightarrow M_n$ induce a group isomorphism $M\cong\lim\limits_{\leftarrow} M_n $. So there is a unique group homomorphism $\sigma: L\rightarrow M$ such that $\sigma$ lifts $\sigma_n$ for any $n\in \Bbb N\cup \{0\}$. In particular $\sigma$ lifts $\bar \sigma$. Since the restriction of $\sigma$ to $P$ and the homomorphism $\varsigma$ both lift the sequence of group homomorphisms $\{\varsigma_n\}_{n\in \Bbb N\cup \{0\}}$, the restriction of $\sigma$ to $P$ has to be equal to $\varsigma$.

We assume that $\sigma'\,\colon L\to M$ is another group homomorphism lifting
$\bar\sigma$ and extending $\varsigma$. Then $\sigma'$ induces a sequence of group homomorphisms $\{\sigma'_n: L\rightarrow M_n\}_{n\in \Bbb N\cup \{0\}}$ such that $\sigma'_{n+1}$ lifts $\sigma'_n$ for any $n\in \Bbb N\cup \{0\}$ and such that $\sigma'_n$ extends $\varsigma_n$ for any $n\in \Bbb N\cup \{0\}$, where $\sigma'_0=\bar\sigma$.
By the uniqueness part of \cite[Lemma 3.6]{PZ2} we find a sequence $\{z_n\}_{n\in \Bbb N\cup \{0\}}$ such that $z_n$ belongs to $Z/Z_n$ for any $n\in \Bbb N\cup \{0\}$, such that $z_n$ is the image of $z_{n+1}$ for any $n\in \Bbb N\cup \{0\}$ and such that for any $n\in \Bbb N\cup \{0\}$, we have $\sigma'_n(x)=\sigma_n(x)^{z_n}$ for any $x\in L$. Since $\lim\limits_{\leftarrow} Z/Z_n\cong Z$, there is some $z\in Z$ lifting $z_n$ for any $n\in \Bbb N\cup \{0\}$. Clearly $\sigma'$ and $\sigma^z$ both lift $\sigma_n'$ for any $n\in \Bbb N\cup \{0\}$ and thus $\sigma'(x)=\sigma(x)^z$ for any $x\in L$.
\end{proof}

\begin{order} Proof {of Lemma \ref{Lift-Cal-L-FrM}}.

{\rm In order to prove this lemma, we apply
Lemma \ref{L-M} to the case that $L=\cal L$, $M=\frak M$, $Z=\frak K$, $\bar\sigma=\frak p$ and $\tau=\frak i$. By \cite[Charpter II, Proposition 8]{S}, there is a canonical group isomorphism $\frak K\cong 1+J(\hat{{{\cal B}}}_{\hat\delta}^Q)$. We identify $\frak K$ with $1+J(\hat{{{\cal B}}}_{\hat\delta}^Q)$. Set ${\frak K}_0=\frak K$ and ${\frak K}_n=1+J(\hat{{{\cal B}}}_{\hat\delta}^Q)^n$ for any $n\in \Bbb N$. Then for any $n\in {\Bbb N}\cup \{0\}$, ${\frak K}_n$ is normal in $\frak M$, ${\frak K}_n/{\frak K}_{n+1}$ is a $p'$-divisible abelian group and canonical group homomorphisms ${\frak K}\rightarrow {\frak K}/{\frak K}_n$ induces a group isomorphism $\lim\limits_{\leftarrow} {\frak K}/{\frak K}_n\cong {\frak K}$ (see \cite[Lemma 45.5]{T}). So the group $\frak K$ together the sequence $\{{\frak K}_n\}_{n\in {\Bbb N}\cup \{0\}}$ satisfies the assumption in Lemma \ref{L-M}.

Let $R$ be a subgroup of $P$ and let $x$ be an element of $\cal L$ such that $R^{x}\subset P$. Since $Q$ is normal in ${\cal L}$, we have $(RQ)^x\leq P$. It is easy to see that if the homomorphisms ${\frak p}$ and $\frak i$ satisfy condition  \hyperlink{cond:5.9.1}{5.9.1} for $RQ$, then the homomorphisms ${\frak p}$ and $\frak i$ satisfy condition  \hyperlink{cond:5.9.1}{5.9.1} for $R$. So we assume that $R$ contains $Q$. Since $P$ is an abelian Sylow $p$-subgroup of ${\cal L}$ (see \cite[Remark 1.9]{KP}), there is $y\in N_{\cal L}(P)$ such that ${u}^{x}={
u}^{y}$ for any $u\in R$ (see \cite[Proposition 49.6]{T}). Set $z=xy^{-1}$. Clearly $z$ belongs to $C_{\cal L}(R)$. So in order to show that the homomorphisms ${\frak p}$ and $\frak i$ satisfy condition \hyperlink{cond:5.9.1}{5.9.1}, it suffices to show that the homomorphisms ${\frak p}$ and $\frak i$ satisfy condition \hyperlink{cond:5.9.1}{5.9.1} for any subgroup $R$ of $P$ containing $Q$ such that $R^{x}=R$.
Let $x'$ be the image of $x$ in ${\cal L}/Q$ and take a representative $a$ of $\lambda_R(x')$ in ${\cal N}_{\hat{\cal A}_{\hat\varepsilon}^*}(R)$. There is a representative $g$ of $\pi(x)$ in $N_G(R_\varepsilon)$ such that $\varphi_{R,\,x^{-1}}^R=\varphi_{R,\,g^{-1}}^R$, such that $a$ belongs to ${\cal N}_{\hat{\cal A}_{\hat\varepsilon}^*}^{\dot g}(R)$ and such that $aua^{-1}=\varphi_{R,\,g^{-1}}^R(u)\hat h$ for any $u\in R$. We denote by $w_x$ the image of $a$ in $\frak M$. Then we have ${\frak i}(u^x)={\frak i}(u)^{w_x}$ for any $u\in R$. Moreover by diagram \ref{eq:dia-5.6.1}, $\frak p$ maps $x$ onto the image of $w_x$ in ${\frak M}/\frak K$. So the homomorphisms ${\frak p}$ and $\frak i$ satisfy condition  \hyperlink{cond:5.9.1}{5.9.1}.

Set $\hat{\cal C}_{\hat\gamma}=\hat i {\cal C}_\gamma \hat i$. Then $\hat {\cal C}_{\hat\gamma}$ is a $\dot J$-graded $P$-interior algebra. As in the proof of Lemma \ref{Iso-Hat-A-ga}, we prove that $\hat {\cal C}_{\hat\gamma}$ is isomorphic to ${\cal O}_*(P\rtimes \hat E_J(P_{\tilde\gamma})^\circ)$ as $\dot J$-graded $P$-interior algebras. In particular, $\hat{{{\cal B}}}_{\hat\gamma}$ is isomorphic to ${\cal O}_*(Q\rtimes \hat E_J(P_{\tilde\gamma})^\circ)$ as $Q$-interior $P$-algebras.
Since $P=Q {\cal S}_P$ (see 3.3), where ${\cal S}_P$ is the subgroup of all $E_J(P_{\tilde\gamma})$-fixed elements of $P$, the $P$-conjugation acts trivially on $({\cal O}_*(Q\rtimes \hat E_J(P_{\tilde\gamma})^\circ))^Q$. So condition  \hyperlink{cond:5.9.2}{5.9.2} is satisfied.

Finally by Lemma \ref{L-M} the homomorphism $\frak p$ can be lifted to a group homomorphism ${\cal L}\rightarrow \frak M$ extending the homomorphism $\frak i$. Let $x$ be an element in the kernel of the homomorphism ${\cal L}\rightarrow \frak M$. Then the isomorphism $\lambda_Q$ maps the image of $x$ in ${\cal L}/Q$ onto $1$. Thus $x$ has to be inside $Q$. Since the homomorphism ${\cal L}\rightarrow\frak M$ extends $\frak i$, $x$ has to be $1$. So the homomorphism ${\cal L}\rightarrow \frak M$ is injective.\qed
}
\end{order}

\begin{order}{\rm We denote by $\frak p'$ the homomorphism ${\cal L}\rightarrow \frak M$ in Lemma \ref{Lift-Cal-L-FrM} and by $\tilde{\cal L}$ the inverse image of ${\frak p'}({\cal L})$ in ${\cal N}_{\hat{\cal A}_{\hat\delta}^*}(Q)$. The group $\tilde{\cal L}$ with the inclusion map $k^*\rightarrow \tilde{\cal L}$ becomes a $k^*$-group with the $k^*$-quotient ${\cal L}$. We endow the twisted group algebra ${\cal O}_*\tilde{\cal L}$ with a $\dot G$-graded algebra structure through isomorphism \ref{eq:5.7.1}. The inclusion map $P\subset \tilde {\cal L}$ induces a $P$-interior algebra structure on ${\cal O}_*\tilde{\cal L}$.
Let ${\cal K}$ the inverse image of $N_H(Q_\delta)/C_H(Q)$  in ${\cal L}$
and let $\tilde {\cal K}$ be the inverse image of ${\cal K}$ in $\tilde {\cal L}$.
}
\end{order}

\begin{lem}\label{Gad-Cal-A-ga-p'} With the notation as above, there is a $\dot{G}$-graded $P$-interior algebra isomorphism
\begin{equation}
{\cal A}_\gamma\cong S\otimes_{\cal O} {\cal O}_*\tilde{\cal L}.
\tag*{5.12.1}
\label{eq:5.12.1}
\end{equation}
\end{lem}

\begin{proof}
The inclusion $\tilde{\cal L}\subset \hat{\cal A}_{\hat\gamma}^*$
induces a $P$-interior algebra homomorphism
${\cal O}_*\tilde{\cal L}\rightarrow \hat{\cal A}_{\hat\gamma}$.
Let $x$ be an element of ${\cal L}$ and let $x'$ be the image of $x$ in ${\cal L}/Q$. We take a representative $a$ of $\lambda_Q(x')$ in ${\cal N}_{\hat{\cal A}_{\hat\varepsilon}^*}(Q)$ and denote by $w_x$ the image of $a$ in $\frak M$. Then $\pi(x)$ has a representative $g$ in $N_G(Q_\delta)$ such that $\varphi_{R,\,x^{-1}}^R=\varphi_{R,\,g^{-1}}^R$, such that $a$ belongs ${\cal N}_{\hat{\cal A}_{\hat\varepsilon}^*}^{\dot g}(R)$ and such that $aua^{-1}=\varphi_{R,\,g^{-1}}^R(u)\hat h$ for any $u\in R$. Clearly the homomorphism $\frak p$ maps $x$ onto the image of $w_x$ in ${\frak M}/\frak K$ and since $a$ belongs to ${\cal N}_{\hat{\cal A}_{\hat\varepsilon}^*}^{\dot g}(R)$, any inverse image of $x$ in $\tilde{\cal L}$ lies in the $\dot g$-component of $\hat{\cal A}_{\hat\gamma}$. Thus the homomorphism ${\cal O}_*\tilde{\cal L}\rightarrow \hat{\cal A}_{\hat\gamma}$ is a $\dot G$-graded algebra homomorphism.

The homomorphism ${\cal O}_*\tilde{\cal L}\rightarrow \hat{\cal A}_{\hat\gamma}$ induces an ${\cal O}$-algebra homomorphism ${\cal O}_*\tilde{\cal K}\rightarrow \hat{{{\cal B}}}_{\hat\gamma}.$
Since $E_H(Q_\delta)$ is a $p'$-group, ${\cal K}$ has a subgroup $E$ such that ${\cal K}=Q\rtimes E$. Thus we have $\tilde{\cal K}=Q\rtimes \tilde E$, where $\tilde E$ is the inverse image of $E$ in $\tilde {\cal K}$.
On the other hand, by Proposition \ref{Graded-Hat-A-ga}, there is a $Q$-interior algebra isomorphism $\hat{{{\cal B}}}_{\hat\gamma}\cong {\cal O}_*(Q\rtimes \hat E_H(Q_\delta)^\circ)$. As in the third paragraph of the proof of Proposition \ref{Graded-Hat-A-ga}, it is easily checked that $\tilde E$ and the inverse image of $\hat E_H(Q_\delta)^\circ$ through this $Q$-interior algebra isomorphism are conjugate in $(\hat{{{\cal B}}}_{\hat\gamma}^Q)^*$.
In particular, this implies that $\hat{{{\cal B}}}_{\hat\gamma}$ is generated by $\tilde E$ and $Q\hat i$. Thus the homomorphism ${\cal O}_*\tilde{\cal K}\rightarrow \hat{{{\cal B}}}_{\hat\gamma}$ is surjective and then is an ${\cal O}$-algebra isomorphism since $\hat {{{\cal B}}}_{\hat\gamma}$ and ${\cal O}_*\tilde K$ have the same ${\cal O}$-rank $|Q||E_H(Q_\delta)|$. Since $\dot{G}$-graded algebras ${\cal O}_*\tilde{\cal L}$ and $\hat{\cal A}_{\hat\gamma}$ are crossed products of $\dot{G}$, the homomorphism ${\cal O}_*\tilde{\cal L}\rightarrow \hat{\cal A}_{\hat\gamma}$ must be an isomorphism.

As in Lemma \ref{Cal-A-ga}, we prove that there is a $\dot G$-graded $P$-interior algebra embedding~${\cal A}_\gamma\rightarrow S\otimes_{\cal O} \hat{\cal A}_{\hat\gamma}$. Then we get a $\dot{G}$-graded $P$-interior algebra embedding
${\cal A}_\gamma\rightarrow S\otimes_{\cal O} {\cal O}_*\tilde{\cal L}. $ Since the identity element of $S\otimes_{\cal O} {\cal O}_*\tilde{\cal L}$ is contained in the unique local point of $P$ on $S\otimes_{\cal O} {\cal O}_*\tilde{\cal K}$ (see Lemma 4.2 above), this embedding must be an isomorphism.
\end{proof}

\begin{order} {\rm Let $\jmath$ be an element of the local point of $Q$ on $S$. By Lemma 4.2, it is easy to check that $\jmath\otimes 1$ is contained in the unique local point of $Q$ on $S\otimes_{\cal O} {\cal O}_*\tilde{\cal K}$ and that the inverse image $j'$ of $\jmath\otimes 1$ in ${\cal A}_\gamma$ through isomorphism \ref{eq:5.12.1} is contained in $\delta$. We adjust $j$ so that $j'=j$. Set ${\cal A}_\delta=j{\cal A} j$ and $V=\jmath S \jmath$. Then isomorphism \ref{eq:5.12.1} induces a $\dot G$-graded $Q$-interior algebra isomorphism ${\cal A}_\delta\cong V\otimes_{\cal O} {\cal O}_*\tilde{\cal L}$, through which, we identify ${\cal A}_\delta$ with $V\otimes_{\cal O} {\cal O}_*\tilde{\cal L}$. On the other hand, the $Q$-interior algebra $V$ is $N_G(Q_\delta)$-stable (see the second paragraph of the proof of Lemma \ref{Inj-map}). Thus for any $x\in N_G(Q_\delta)$ there is an invertible element $s_x$ in $V$ such that $s_x u s_x^{-1}=\varphi_{Q,\,x}^Q(u)1_V$ for any $u\in Q$. Then $\varphi_{Q,\,x}^Q$ belongs to $F_V(Q_{\{1_V\}})$ and the map ${E}_{G, \, \dot G}(Q_\delta)\rightarrow F_V(Q_{\{1_V\}})$ sending $(\varphi_{Q,\,x}^Q,\,\dot{x})$ onto $\varphi_{Q,\,x}^Q$ for any $x\in N_G(Q_\delta)$ is a group homomorphism. Moreover by \cite[2.12.4]{KP}, this group homomorphism can be lifted to a group homomorphism $\theta:{E}_{G, \, \dot G}(Q_\delta)\rightarrow \hat F_V(Q_{\{1_V\}})$. We set $\theta(\varphi_{Q,\,x}^Q,\,\dot{x})=(\varphi_{Q,\,x}^Q,\, \bar u_x)$. For any $\tilde y\in \tilde {\cal L}$, we take a representative $z$ of $\pi(y)$ in $N_G(Q_\delta)$. Then it is easy to check that the correspondence
\begin{equation}
\tilde{\cal L}\rightarrow \hat F_{{\cal A},\,\dot G}(Q_\delta),\, \tilde y\mapsto (\varphi_{Q,\,z}^Q,\,\dot{z},\, \overline{u_z\otimes \tilde y})
\tag*{5.13.1}
\label{eq:5.13.1}
\end{equation}
is a $k^*$-group homomorphism, where $\overline{u_z\otimes \tilde y}$ is the image of $u_z\otimes \tilde y$ in $\hat {\cal F}_{\cal A}(Q_\delta)$.}
\end{order}

\begin{lem}\label{k*-Hat-Hom} With the above notation, the inclusion ${E}_{G, \, \dot G}(Q_\delta)\subset F_{{\cal A},\,\dot G}(Q_\delta)$ can be lifted to an injective $k^*$-group homomorphism $\hat{E}_{G, \, \dot G}(Q_\delta)^\circ\rightarrow \hat F_{{\cal A},\,\dot G}(Q_\delta)$.
\end{lem}

 \begin{proof}For an element $\overline{(x, \, s_\delta(a))}$ in $\hat{E}_{G, \, \dot G}(Q_\delta)^\circ$, we have $s_\delta(j^{xa^{-1}})=s_\delta(j)$ and $s_\epsilon(j^{xa^{-1}})=s_\epsilon(j)$ for any other point $\epsilon$ of $Q$ on $B$. Then by \cite[Lemma 6.3]{P6} there is a suitable element $c$ of $1+J({{{\cal B}}}^Q)$ such that
$xa^{-1}c^{-1}$ and $j$ commute each other and such that $j(xa^{-1})j=(xa^{-1}c^{-1})j$. In particular, $j(xa^{-1})j$ belongs to ${\cal N}_{{\cal A}_\delta^*}(Q)$. For any $a\in {\cal N}_{{\cal A}_\delta^*}(Q)$, we denote by $\bar a$ the image of $a$ in $\hat {\cal F}_{\cal A}(Q_\delta)$. We define a correspondence $\hat\theta_\delta: \hat{E}_{G, \, \dot G}(Q_\delta)^\circ\rightarrow \hat F_{{\cal A},\,\dot G}(Q_\delta)$ sending $\overline{(x, \, s_\delta(a))}$ to $(\varphi_{Q,\,x}^Q,\,\dot x,\,\overline{j(xa^{-1})j})$.

Let $a'$ be another invertible element of ${{{\cal B}}}^Q$ such that $s_\delta(a')=s_\delta(a)$. By \cite[Lemma 6.3]{P6} again, there is a suitable element $c'$ of $1+J({{{\cal B}}}^Q)$ such that $xa'^{-1}c'^{-1}$ and $j$ commute and such that $j(xa'^{-1})j=(xa'^{-1}c'^{-1})j$. Clearly we have $((xa'^{-1}c'^{-1})j)^{-1}(xa^{-1}c^{-1})j=c'a'a^{-1}c^{-1}j$ and $c'a'a^{-1}c^{-1}j$ belongs to $1+J({{{\cal B}}}_\delta^Q)$. Thus we have  $\overline{j(xa^{-1})j}=\overline{xa^{-1}c^{-1}j}=\overline{xa'^{-1}c'^{-1}j}
=\overline{j(xa'^{-1})j}$ and $(\varphi_{Q,\,x}^Q,\,\dot x,\,\overline{j(xa^{-1})j})=(\varphi_{Q,\,x}^Q,\,\dot x,\,\overline{j(xa'^{-1})j})$. If $(y, \, s_\delta(d))\in \hat N_G(Q_\delta)$ has the same image as $(x, \, s_\delta(a))$ in $\hat{E}_{G, \, \dot G}(Q_\delta)^\circ$, there is some $z\in C_H(Q)$ such that $s_\delta(d)=s_\delta(az)$ and $y=xz$. As above we choose $d$ to be $az$. Then we have $(\varphi_{Q,\,x}^Q,\,\dot x,\,\overline{j(xa^{-1})j})=(\varphi_{Q,\,y}^Q,\,\dot y,\, \overline{j(yd^{-1})j})$ and so $\hat\theta_\delta$ is a well defined map.

We take two elements $(x, \, s_\delta(a))$ and $(y, \, s_\delta(d))$ in $\hat N_G(Q_\delta)$ such that $xa^{-1}$ and $yd^{-1}$ commute with $j$.
We take an element $c''$ of $1+J({{{\cal B}}}^Q)$ such that $xy (ad)^{-1}c''^{-1}$ and $j$ commute and such that $jxy (ad)^{-1}j=hxy (ad)^{-1}c''^{-1}$. Since $s_\delta(a^{yd^{-1}})=s_\delta(c''a)$,
we have
\begin{eqnarray*}\hat\theta_\delta(\overline{(x, \, s_\delta(a))}\,\overline{(y, \, s_\delta(d))})&=&\hat\theta_\delta(\overline{(xy, \, s_\delta(ad))})=(\varphi_{Q,\,xy}^Q,\,\dot{x}\dot{y},\, \overline{xy(ad)^{-1}c''^{-1}j})\cr
&=&(\varphi_{Q,\,xy}^Q,\,\dot{x}\dot{y},\, \overline{xa^{-1}j}\,\,\overline {yd^{-1}a^{yd^{-1}}a^{-1}c''^{-1}j})\cr
&=&(\varphi_{Q,\,x}^Q,\,\dot x,\,\overline{xa^{-1}j})(\varphi_{Q,\,y}^Q,\,\dot{y},\,\overline {yd^{-1}j})=\hat\theta_\delta(\overline{(x, \, s_\delta(a))})\hat\theta_\delta(\overline{(y, \, s_\delta(d))})
\end{eqnarray*}
and thus $\hat\theta_\delta$ is a group homomorphism. Clearly $\hat\theta_\delta$ maps $\overline{(x,\, s_\delta(\lambda a))}$ onto $(\varphi_{Q,\,x}^Q,\,\dot x,\,\overline{j(x\lambda a^{-1})j})$ for any $\lambda\in {\cal O}^*$ and lifts the inclusion ${E}_{G, \, \dot G}(Q_\delta)\subset F_{{\cal A},\,\dot G}(Q_\delta)$. So $\hat\theta_\delta$ is an injective $k^*$-group homomorphism.
\end{proof}

\begin{lem}\label{k*-Iso-HatL-TildeL} With the notation as above, there is a $k^*$-group isomorphism $\hat{\cal L}^\circ\cong \tilde{\cal L}$, which lifts the identity map on ${\cal L}$.
\end{lem}

 \begin{proof}By Lemma \ref{k*-Hat-Hom}, we get an injective $k^*$-group homomorphism $\hat{E}_{G, \, \dot G}(Q_\delta)^\circ\rightarrow \hat F_{{\cal A},\,\dot G}(Q_\delta)$ lifting
the inclusion ${E}_{G, \, \dot G}(Q_\delta)\subset F_{{\cal A},\,\dot G}(Q_\delta)$. Moreover by the proof of Lemma \ref{k*-Hat-Hom}, the image of this $k^*$-group homomorphism coincides with the image of homomorphism \ref{eq:5.13.1}. Then by factoring it through homomorphism \ref{eq:5.13.1}, we get a $k^*$-group homomorphism $\tilde {\cal L}\rightarrow \hat{E}_{G, \, \dot G}(Q_\delta)^\circ$, which clearly lifts $\pi$. The proof is done.
\end{proof}

\begin{thm}\label{Iso-G-ga} Keep the notation as above and assume that $P$ is abelian and that the block $b$ of $H$ is inertial. Then there is a determinant one $P$-interior full matrix algebra $S$ over ${\cal O}$ such that we have a $\dot G$-graded $P$-interior algebra isomorphism (see paragraph \ref{Notation2})
\begin{equation}
({\cal O} G)_\gamma\cong S\otimes_{\cal O} {\cal O}_*\hat {\cal L}^\circ.
\tag*{5.16.1}
\label{eq:5.16.1}
\end{equation}
 Moreover, $S$ has a $P$-stable ${\cal O}$-basis containing the unity of $S$
and $S$ is unique up to $P$-interior algebra isomorphisms.
\end{thm}

\begin{proof} By Lemmas \ref{Gad-Cal-A-ga-p'} and \ref{k*-Iso-HatL-TildeL}, we get the isomorphism \ref{eq:5.16.1}. Since the restriction to ${\cal C}_\gamma$ of the isomorphism \ref{eq:5.16.1} induces a $\dot{J}$-graded $P$-interior algebra isomorphism
${\cal C}_\gamma\cong S\otimes_{\cal O} {\cal O}_*(P\rtimes \hat E_J(P_{\tilde\gamma})^\circ)$, the last statement of Theorem \ref{Iso-G-ga} follows from the last statement of paragraph \ref{NotationF5}.
\end{proof}

\begin{order} Proof of Theorem \ref{Main1}.

{\rm We identify ${\cal A}_\gamma$ with $S\otimes_{\cal O} {\cal O}_*\hat{\cal L}^\circ$ through isomorphism \ref{eq:5.16.1}.
Let $V$ be an ${\cal O} P$-module such that ${\rm End}_{\cal O}(V)\cong
S$.
Clearly the $({\cal
A}_\gamma\otimes_{\cal O} {\cal O}_*\hat {\cal L})$-module $V\otimes_{\cal O} {\cal O}_*\hat
{\cal L}^\circ$ induces a Morita equivalence from ${\cal A}_\gamma$
to~${\cal O}_*\hat {\cal L}^\circ$. Since the $({\cal A}\otimes_{\cal O} {\cal
A}_\gamma^\circ)$-module ${\cal A}i$ induces a Morita equivalence
from $\cal A$ to ${\cal A}_\gamma$, the $({\cal A}\otimes_{\cal O}
{\cal O}_*\hat {\cal L})$-module
${\cal A}i\otimes_{{\cal A}_\gamma} (V\otimes_{\cal O} {\cal O}_*\hat {\cal L}^\circ)
\cong {\cal A}i\otimes_{S} V$
 induces a Morita equivalence from  $\cal A$ to~${\cal O}_*\hat {\cal L}^\circ\,.$

Let $\imath$ and $\jmath$ be elements of local points of $P$ and $Q$ on $S$, respectively. By Lemma \ref{Iso-Hat-A-ga}, it is easy to checked that $\imath\otimes 1$ and $\jmath\otimes 1$ are contained in the unique local points of $P$ and $Q$ on $S\otimes_{\cal O} {\cal O}_*\hat{\cal K}^\circ$, respectively. Therefore we can identify $S$ with a subalgebra of ${\cal A}_\gamma$ through  isomorphism \ref{eq:5.16.1} so that $i$ and $j$ are inside $S$. Set $V_\delta = j(V)$ and $S_\delta = {\rm End}_{\cal O} (V_\delta)$. The
$({{{\cal B}}}\otimes_{\cal O} {\cal O}_*\hat {\cal K})$-module
${{{\cal B}}}j\otimes_{{{{\cal B}}}_\delta} (V_\delta\otimes_{\cal O} {\cal O}_*\hat {\cal K}^\circ)
\cong {{{\cal B}}}j\otimes_{S_\delta} V_\delta$ induces a Morita
equivalence from ${{\cal B}}$ to ${\cal O}_*\hat {\cal K}^\circ\,.$

Since we have isomorphism \ref{eq:3.10.1}, analogously with the evident notation, the $({\mathbbm A}\otimes_{\cal O}
{\cal O}_*\hat {\cal L})$-module ${\mathbbm A} {\mathbbm i} \otimes_{{\mathbbm S}}
{\mathbbm V}$ induces a Morita equivalence from ${\mathbbm A}$
to~${\cal O}_*\hat {\cal L}^\circ$, whereas the $({{\mathbbm B}}\otimes_{\cal O}
{\cal O}_*\hat {\cal K})$-module ${{\mathbbm B}}{\mathbbm j} \otimes_{{\mathbbm S}_{\bbdelta}} {\mathbbm V}_{\bbdelta}$
 induces a Morita equivalence from  ${\mathbbm B}$ to~${\cal O}_*\hat {\cal K}^\circ\,.$

 \smallskip
 Consequently, the ${\cal O} (G\times {\mathbbm G})$-module
 $D= ({\cal A}i\otimes_{S} V)\otimes_{{\cal O}_*\hat {\cal L}^\circ}
  ({\mathbbm V}^\circ \otimes_{{\mathbbm S}} \i {\mathbbm A})$
 induces a Morita equivalence from  $\cal A$ to~${\mathbbm A}\,,$
 whereas the ${\cal O} (H\times {\mathbbm H})$-module
 $M = ({{{\cal B}}}j\otimes_{S_\delta} V_\delta)\otimes_{{\cal O}_*\hat {\cal K}^\circ}
 ({\mathbbm V}_{\bbdelta}^\circ \otimes_{{\mathbbm S}_\delta} \j {{\mathbbm B}})$
induces a Morita equivalence from  ${{\cal B}}$ to~${\mathbbm B}\,.$

 \smallskip
 Moreover, since we have the obvious inclusions
 ${{{\cal B}}}j\subset {\cal A}i$, $S_\delta \subset S$ and
 $V_\delta \subset V$,
it is easily checked that we have
\begin{equation}
{{{\cal B}}}j\otimes_{S_\delta} V_\delta\cong {{{\cal B}}}i\otimes_{S} V\subset {\cal A}i\otimes_{S} V;
\tag*{5.17.1}
\label{eq:5.17.1}
\end{equation}
 in particular, we have an evident section
${\cal A}i\otimes_{S} V\longrightarrow {{{\cal B}}}j\otimes_{S_\delta} V_\delta$
which is actually a $({{{\cal B}}}\otimes_{\cal O} {\cal O}_*\hat {\cal K})$-module
homomorphism. Similarly, we have a split $({{\mathbbm B}} \otimes_{\cal O}
{\cal O}_*\hat {\cal K})$-module monomorphism
\begin{equation}
{{\mathbbm B}}\mathbbm{j} \otimes_{{\mathbbm S}_\delta} {\mathbbm V}_\delta\longrightarrow
{\mathbbm A} \mathbbm{i} \otimes_{{\mathbbm S}} {\mathbbm V} .
\tag*{5.17.2}
\label{eq:5.17.2}
\end{equation}

 \smallskip
 In conclusion, the $({{{\cal B}}}\otimes_{\cal O} {\cal O}_*\hat {\cal K})$- and
 $({{\mathbbm B}}\otimes_{\cal O} {\cal O}_*\hat {\cal K})$-module homomorphisms~ \ref{eq:5.17.1} and~ \ref{eq:5.17.2},
 together with the inclusion ${\cal O}_*\hat {\cal K}\subset {\cal O} \hat {\cal L}\,,$ determine
  an  ${\cal O} (H\times {\mathbbm H})$-module homomorphism
 \begin{equation}
M\longrightarrow {\rm Res}_{H\times {\mathbbm H}}^{G\times {\mathbbm G}} (D)
\tag*{5.17.3}
\label{eq:5.17.3}
\end{equation}
 which actually admits a section too. Now
 we claim that the product by $K$ stabilizes the image of $M$ in $D\,,$ so that
 $M$ can be extended to an ${\cal O} K$-module.

 \smallskip
 It suffices to prove that the image of $M$ is stabilized by the multiplication
 by $\Delta (N_G(Q_\delta))\,.$ Given $x\in N_G(Q_\delta)$, there are
some invertible elements $a_x\in ({\cal O} H)^Q$ and  $b_x\in ({\cal O}
{\mathbbm H})^Q$ such that
$xjx^{-1} = a_xj a_x^{-1}$ and $x {\mathbbm j} x^{-1} = b_x {\mathbbm j}  b_x^{-1}$
and therefore $a_x^{-1}x$ and $b_x^{-1}x$ respectively centralize
$j$ and ${\mathbbm j} \,,$ so that $a_x^{-1}xj$ and $b_x^{-1}x {\mathbbm j} $
respectively belong to ${\cal A}_\delta$ and to ${\mathbbm A}_{\bbdelta};$ but, according to isomorphisms ~\ref{eq:5.16.1} and \ref{eq:3.10.1}, we
have $\dot{G}$-graded $Q$-interior algebra isomorphisms
${\cal A}_\delta\cong S_\delta\otimes_{\cal O} {\cal O}_*\hat {\cal L}^\circ$ and $
{\mathbbm A}_{\bbdelta}\cong {\mathbbm S}_{\bbdelta}\otimes_{\cal O} {\cal O}_*\hat
{\cal L}^\circ
$
where we are setting ${\mathbbm S}_{\bbdelta} =
{\mathbbm j} {\mathbbm S}{\mathbbm j} \,.$

\smallskip
Hence, identifying  with each other both members of these
isomorphisms and modifying if necessary our choice of $a_x\,,$ for
some $s_x\in S_\delta\,,$ $t_x\in {\mathbbm S}_{\bbdelta}$ and $\hat
y_x\in \hat {\cal L}\,,$
 we get
$a_x^{-1}x j = s_x\otimes \hat y_x$ and $b_x^{-1}x {\mathbbm j}  = t_x\otimes \hat y_x.$ Thus, setting ${\mathbbm V}_\delta = {\mathbbm j} ({\mathbbm V})\,,$
for any $a\in {{{\cal B}}}j\,,$ any $d\in {{\mathbbm B}}{\mathbbm j} \,,$ any
$v\in V_\delta$ and any $w\in {\mathbbm V}_{\bbdelta}\,,$ in $D$ we have
\begin{eqnarray*}(x,x)\!\!\!\!&\cdot&\!\!\!\!(a\otimes v)\otimes (w\otimes d)
= (x a\otimes v)\otimes (w\otimes dx^{-1})\\
&=& (x ax^{-1}a_x (a_x^{-1}xj)\otimes v)\otimes
(w\otimes ({\mathbbm j}  x^{-1}b_x)b_x^{-1}xdx^{-1})\\
&=&(x ax^{-1}a_x \otimes s_x\cdot v)\cdot \hat y_x\otimes
\hat y_x^{-1}\cdot(w\cdot t_x^{-1}\otimes b_x^{-1}xdx^{-1})\\
&=&(x ax^{-1}a_x \otimes s_x\cdot v) \otimes (w\cdot t_x^{-1}\otimes
b_x^{-1}xdx^{-1});
\end{eqnarray*}
since $x ax^{-1}a_x $ and $b_x^{-1}xdx^{-1}$ respectively belong to
${{{\cal B}}}j$ and ${\mathbbm j} {{\mathbbm B}}\,,$ this proves our claim.

\smallskip
Finally, since homomorphism~\ref{eq:5.17.3} actually becomes an ${\cal O}
K$-module homomorphism, it induces an ${\cal O} (G\times {\mathbbm G})$-module
homomorphism
${\rm Ind}_K^{G\times {\mathbbm G}}(M)\longrightarrow D$,
which is actually an isomorphism as it is easily checked.\qed}
\end{order}

\begin{order} Proof of Corollary \ref{Main2}.

{\rm
Assume that $P$ is abelian and that  the block $b$  is inertial as a block of $PH$.
Let $N$ be a normal subgroup of $G$ containing $H$ such that $G=PN$ and $N/H$ is a $p'$-group.

Since $N/H$ and $N_H(Q_{\delta})/C_H(Q)$ are both $p'$-groups, so is
$N_N(Q_{\delta})/C_H(Q)$. Every block of $N_N(Q_{\delta})$ covering $b_{\delta}$ has $Q$ as a defect group. Set $C_N(Q_{\delta})=N_N(Q_{\delta})\cap C_N(Q)$.
Then $C_N(Q_{\delta})/C_H(Q)$ is a $p'$-group.
So all blocks of $C_N(Q_{\delta})$ covering $b_{\delta}$ have $Q$ as defect group and are nilpotent.
Since $P$ is abelian, we have $C_G(Q_{\delta})=PC_N(Q_{\delta})$.
By \cite[Theorem 8.12.2]{L}, every block of $C_G(Q_{\delta})$ covering $b_{\delta}$ is also nilpotent.
Since $C_G(Q_{\delta})$ is normal in $N_G(Q_{\delta})$ and $N_G(Q_{\delta})/C_G(Q_{\delta})\cong N_N(Q_{\delta})/C_N(Q_{\delta})$ is a $p'$-group,
by \cite[Corollary]{Z2}, every block of $N_G(Q_{\delta})$ covering $b_{\delta}$ is inertial, thus every block of $N_G(Q)$ covering $b_{\delta}$ is inertial.
For any  block $c$ of $G$ covering $b$,  let ${\mathbbm c}$ be the corresponding block of $N_G(Q)$ such that the the block algebra ${\cal O} Gc$ and ${\cal O} N_G(Q)\mathbbm c$
are  basically Morita equivalent. Clearly $\mathbbm c$ covers $b_{\delta}$.
So $\mathbbm c$ is inertial and so is $c$.\qed}
\end{order}

\end{document}